\documentclass[12pt]{amsart}

\usepackage{amsmath,amssymb}
\usepackage{amsfonts}
\usepackage{amsthm}
\usepackage{latexsym}
\usepackage{graphicx}
\usepackage{epsfig}

\newtheorem{theorem}{Theorem}[section]
\newtheorem{lemma}[theorem]{Lemma}
\newtheorem{proposition}[theorem]{Proposition}
\newtheorem{definition}[theorem]{Definition}
\newtheorem{corollary}[theorem]{Corollary}

\newtheorem{prop}{Proposition}[section]
\newtheorem{remark}[prop]{Remark}

\makeatletter \@addtoreset{equation}{section} \makeatother

\def\ddt{\frac{d}{dt}}
\def\ppt{\frac{\partial}{\partial t}}

\def\RR{{\mathrm R}}
\def\WW{{\mathrm W}}

\def\EE{{\mathrm E}}
\def\HH{{\mathrm H}}

\def\Rc{{\mathrm {Rc}}}

\def\SS{{\mathrm S}}

\def\Rm{{\mathrm {Rm}}}
\def\He{\mathrm {Hess}}

\begin{document}

\title{The Weyl Tensor of Gradient Ricci Solitons}

\author{Xiaodong Cao$^*$}

\address{Department of Mathematics,
 Cornell University, Ithaca, NY 14853-4201}
\email{cao@math.cornell.edu, hungtran@math.cornell.edu}
\thanks{$^*$This work was partially supported by grants from the Simons Foundation (\#266211 and \#280161)}

\author{Hung Tran}


\renewcommand{\subjclassname}{%
  \textup{2000} Mathematics Subject Classification}

\date{\today}



\begin{abstract} This paper derives new identities for the Weyl tensor on a gradient Ricci soliton, particularly in dimension four. First, we prove a Bochner-Weitzenb\"ock type formula for the norm of the self-dual Weyl tensor and discuss its applications, including connections between geometry and topology. In the second part, we are concerned with the interaction of different components of Riemannian curvature and (gradient and Hessian of) the soliton potential function. The Weyl tensor arises naturally in these investigations.  Applications here are rigidity results. 
\end{abstract}
\maketitle
\tableofcontents
\markboth{Xiaodong Cao and Hung Tran} {Weyl Tensor of Gradient Ricci Solitons}

\section{\textbf{Introduction}}

The Ricci flow,  which was first introduced by R. Hamilton in \cite{H3}, describes a one-parameter family of smooth metrics $g(t)$, $0\leq t<T\leq \infty$, on a closed n-dimensional manifold  $M^n$, by  the equation
 \begin{equation}
 \label{RF}
 \ppt g(t)=-2\text{Rc}(t). \end{equation}
 The subject has been studied intensively, particularly in the last decade thanks to seminal contributions by G. Perelman in his proof of the Poincar\'e conjecture (cf. \cite{perelman1, perelman2}). It also gains more popularity after playing a key role in the proofs of the classification theorem for  manifolds with 2-positive curvature operators by C. B\"ohm and B. Wilking \cite{bohmwilking}, and  the Differentiable Sphere Theorem of S. Brendle and R. Schoen \cite{  bs091, bs072}.\\
 
 As a weakly parabolic system, the Ricci flow can develop finite-time singularities and, consequently, the study of singularity models becomes essentially crucial. In this paper, we are concerned with gradient Ricci solitons (GRS), which are self-similar solutions of Hamilton\rq{}s Ricci flow (\ref{RF}) and arise naturally in the analysis of singularities. A GRS $(M,g, f, \lambda)$ is a Riemannian manifold  endowed with a special structure given by a (soliton) potential function $f$, a constant $\lambda$, and the equation 
 \begin{equation}
 \label{grs}
 \Rc+\text{Hess}{f}=\lambda g. \end{equation}
 Depending on the sign of $\lambda$, a GRS is called shrinking (positive), steady (zero), or expanding (negative). In particular an Einstein manifold N can be considered as a special case of a GRS where $f$ is a constant and $\lambda$ becomes the Einstein constant. A less trivial example is a Gaussian soliton $(\mathbb{R}^{k}, g_{\text{sd}}, \lambda \frac{|x|^2}{2},\lambda)$ with $g_{\text{sd}}$ being the standard metric on Euclidean space. It is interesting to note that $\lambda$ can be an arbitrary real number and that the Gaussian soliton can be either shrinking, steady or expanding. Furthermore, a combination of those two above, by the notation of P. Petersen and W. Wylie \cite{pewy09}, is called a \textit{rank k rigid GRS}, namely a quotient of $N\times \mathbb{R}^k$. Other nontrivial examples of GRS are rare and mostly K\"ahler, see \cite{caohd96, fik03}.\\
 
 In recent years, following the interest in the Ricci flow, there have been various efforts to study the geometry and classification of  GRS's; for example, see \cite{caohd09} and the citations therein. In particular, the low-dimensional cases ($n=2, 3$) are relatively well-understood. For $n=2$,  Hamilton \cite{Hsurface} completely classified shrinking gradient solitons with bounded curvature and showed that they must be either the round sphere, projective space, or Euclidean space with standard metric. For $n=3$, utilizing the Hamilton-Ivey estimate, Perelman \cite{perelman2} proved an analogous theorem. Other significant results include recent development of Brendle \cite{b12rot} showing that a non-collapsed steady GRS must be rotationally symmetric and is, therefore, isometric to the Bryant soliton. \\

In higher dimensions, the situation is more subtle mainly due to the non-triviality of the Weyl tensor ($\WW$) which is vacuously zero for dimension less than four. One general approach to the classification problem so far has been imposing certain restrictions on the curvature operator. An analogue of Hamilton-Perelman results was obtained by A. Naber proving that a four dimensional complete non-compact GRS with bounded nonnegative curvature operator must be a finite quotient of $\RR^4$,  $S^2\times \RR^2$ or $S^3\times \RR$ \cite{naber07}. In \cite{kb08}, B. Kotschwar classified all rotationally symmetric GRS's with given diffeomorphic types on $\RR^n$, $S^{n-1}\times \RR$ or $S^n$. Note that any rotationally symmetric Riemannian manifold has vanishing Weyl tensor. 
 
Thus, a natural development is to impose certain conditions on that Weyl tensor. If the dimension is at least four, then a complete shrinking GRS with vanishing Weyl tensor must be a finite quotient of $\RR^n$, or $S^{n-1}\times \RR$ or $S^n$  following the works of \cite{niwa08, zhangzh09, cw07, pewy10}; a steady GRS is flat or rotationally symmetric (that is, a Bryant Soliton) by \cite{caochen12}. The assumption $\WW\equiv 0$ can be weakened to $\delta\WW\equiv 0$, a closed or non-compact shrinking GRS must be rigid \cite{cw07, fega11, munse09}; or in dimension four, to the vanishing of self-dual Weyl tensor only, a shrinking GRS with bounded curvature must be a finite quotient of $\RR^4$, $S^{3}\times \RR$, $S^n$, or $CP^2$, and steady GRS must be a Bryant soliton or flat \cite{cw11}. There are some other classifications based on, for instance, Bach flatness \cite{caochen11} or assumptions on the radial sectional curvature \cite{pewy10}.\\
 
As the major obstruction to understand  GRS in higher dimensions is the non-triviality of the Weyl tensor, this paper is devoted to studying the delicate role of the Weyl tensor within a gradient soliton structure. Our perspective here is to view a GRS as both  generalization of an Einstein manifold as well as a self-similar solution to the Ricci flow. 
In particular, this paper derives several new identities on the Weyl tensor of GRS in dimension four. In the first part, we prove the following Bochner-Weitzenb\"ock type formula for the norm of the self-dual Weyl tensor using flow equations and some ideas related to Einstein manifolds. 

\begin{theorem} 
\label{WSoliton}
Let $(M,g,f,\lambda)$ be a four-dimensional GRS. Then we have the following Bochner-Weitzenb\"ock formula:
\begin{align}\label{Wsolitoneqn}
\Delta_{f} |\WW^{+}|^2=&2|\nabla \WW^{+}|^2+4\lambda |\WW^{+}|^2-36\text{det}\WW^{+} -\left\langle{\Rc\circ \Rc,\WW^{+}}\right\rangle\nonumber\\
=&2|\nabla \WW^{+}|^2+4\lambda |\WW^{+}|^2-36\text{det}\WW^{+} -\left\langle{\He{f}\circ\He{f},\WW^{+}}\right\rangle .
\end{align}
\end{theorem}

For the relevant notation, see Section 2. Identity (\ref{Wsolitoneqn}) potentially has several applications and we will present a couple of them in this paper including a gap theorem. More precisely, if the GRS is not locally conformally flat and the divergence of the Weyl tensor is relatively small, then the $L_{2}$-norm of the Weyl tensor is bounded below by a topological constant (cf. Theorem \ref{gapW}). The proof, in a similar manner to that of \cite{g00}, uses some ideas from the solution to the Yamabe problem.\\

In the second part, we are mostly concerned with the interaction of different curvature components, gradient and Hessian of the potential function. In particular, an interesting connection is illustrated by the following integration by parts formula.

\begin{theorem}
\label{intanddiv}
Let $(M,g,f,\lambda)$ be a closed GRS. Then we have the following identity:
\begin{align}
\int_{M}\left\langle{\WW, \Rc\circ \Rc}\right\rangle&= \int_{M}\left\langle{\WW, \He{f}\circ \He{f}}\right\rangle= \int_{M}\WW(\text{Hess}f, \text{Hess}f)=\int_{M}\WW_{ijkl}f_{ik}f_{jl}\nonumber\\
&=\frac{1}{n-3}\int_{M}\left\langle{\delta \WW, (n-4)M+(n-2)P}\right\rangle. 
\end{align}
In particular, in dimension four, the identity becomes
\begin{equation}
\int_{M}\left\langle{\WW, \Rc\circ \Rc}\right\rangle=4\int_{M}|\delta \WW|^2.
\end{equation}
\end{theorem}

\begin{remark}
For definitions of $M$ and $P$, see Section 5. 
\end{remark}

\begin{remark}
 This result exposes the intriguing interaction between the Weyl tensor and the potential function f on a GRS. It will be interesting  to extend those identities to a (possibly non-compact) smooth metric measure space or generalized Einstein manifold.
 \end{remark}
 
\begin{remark}
\label{intanddivplus}
In dimension four, the statement also holds if replacing $\WW$ by $\WW^{\pm}$, see Corollary \ref{intanddivpluscor}.
\end{remark}  

The interactions of various curvature components and the soliton potential function can be applied to study the classification problem. For example, Theorem \ref{rigidRc} asserts rigidity of the Ricci curvature tensor in dimension four. More precisely, if the Ricci tensor at each point has at most two eigenvalues with multiplicity one and three, then any such closed GRS must be rigid. It is interesting to compare this result with classical classification results of the Codazzi tensor, which requires both  distribution of eigenvalues and information on the first derivative (see \cite[Chapter 16, Section C]{besse}). \\

This paper is organized as follows. In Section 2, we fix our notation and collect some preliminary results. Section 3 provides a proof of Theorem \ref{WSoliton} using the Ricci flow technique. Section 4 gives some immediate applications of the new Bochner-Weitzenb\"ock type formula including the aforementioned gap theorem. In Section 5, we first discuss a general framework to study the interaction of different components of the curvature with the potential function, and then prove Theorem \ref{intanddiv}. In Section 6,  we apply our framework to obtain various rigidity results. Finally, in the Appendix, we collect a few related formulas.

\section{\textbf{Notation and Preliminaries}}
In this section, we will fix the notation and convention that will be used throughout the paper.\\

$\RR, \WW, \Rc, \SS, \EE$ will stand for the Riemannian curvature operator, Weyl tensor, Ricci curvature, scalar curvature, and the traceless part of the 
Ricci respectively. 

For a finite dimensional real vector space (bundle), $\Lambda_{2}(V)$ denotes the space of bi-vectors or two-forms. In our case,  the space of interest is normally the tangent bundle and when the context is clear, the dependence on $V$ is omitted. 
    
Given an orthonormal basis  $\{E_{i}\}_{i=1}^{n}$ of $T_p M$, it is well-known that we can construct an  orthonormal frame about p such that $e_i(p)=E_i$ and $\nabla e_i \mid_p =0$. Such a frame is called normal at p. Also $e_{12}$ is the shorthand notation for $e_{1}\wedge e_{2}\in \Lambda_{2}$.  

The modified Laplacian is defined as 
\begin{equation*} \Delta_{f}=\Delta-\nabla_{\nabla f}.
\end{equation*}

For any $(m,0)$-tensor $T$, its divergence operator is defined as
\[(\delta T)_{p_{2}...p_{m}}=\sum_{i} \nabla_{i}T_{i p_{2}...p_{m}}, \]
while its interior product by a vector field $X$ is defined as
\[(i_{X} T)_{p_{2}...p_{m}}=T_{X p_{2}...p_{m}}. \]

Furthermore, we will interchange the perspective of a vector and a covector freely, i.e., a $(2,0)$ tensor will also be seen as a $(1,1)$ tensor. Similarly, a $(4,0)$ tensor such as $\RR, \WW$ can be interpreted as an operator on bi-vectors, that is, a map from $\Lambda_{2}(TM)\rightarrow \Lambda_{2}(TM)$. Consequently, the norm of  these operators is agreed to be  sum of all eigenvalues squared (this agrees with the tensor norm defined in \cite{chowluni} for $(2,0)$ tensors but differs by $1/4$-factor for $(4,0)$ tensors). More precisely, 
\[|\WW|^2=\sum_{i<j; k<l}\WW_{ijkl}^{2}.\]

In addition, the norm of covariant derivative and divergence on these tensors can be defined accordingly, 
\begin{align*}
|\nabla \WW|^{2}&=\sum_{i}\sum_{a<b; c<d}(\nabla_{i}\WW_{abcd})^2,\\
|\delta \WW|^2&=\sum_{i}\sum_{a<b}((\delta\WW)_{iab})^2.
\end{align*} 

For a tensor $T:\Lambda_2 (TM)\otimes (TM)\rightarrow \mathbb{R}$, we define
\begin{align}
\label{productconv}
 \left\langle{T,\delta \WW}\right\rangle=\sum_{i<j;k} T_{ijk}(\delta\WW)_{kij}, \\
 \left\langle{T,i_{X} \WW}\right\rangle=\sum_{i<j;k} T_{ijk}(i_{X}\WW)_{kij}.
 \end{align}
Finally, when the context is clear, we will omit the measure when integrating.

\subsection{Gradient Ricci Solitons} 

In this subsection, we recall some well-known identities for GRS's. A GRS is characterized by the Ricci soliton equation
\begin{equation}
\label{GRS}
\text{Rc}+\nabla \nabla f = \lambda g.
\end{equation}

Algebraic manipulation of (\ref{GRS}) and application of the Bianchi identities lead to following formulas (for a proof see \cite{chowluni}), 
\begin{align}
\SS + \triangle f &= n\lambda,\\
\label{rcdiv}
\frac{1}{2}\nabla_{i}\SS=\nabla^{j}\RR_{ij} &= \RR_{ij}\nabla^{j}f,\\
\label{rcandf}
\Rc(\nabla{f})&=\frac{1}{2}\nabla{\SS},\\
\SS+|\nabla f|^2-2\lambda f &= \text{constant},\\
\label{lapS}
\triangle\SS+2|\text{Rc}|^2 &= \left\langle{\nabla f,\nabla \SS}\right\rangle+2\lambda \SS.
\end{align}

\begin{remark}
If $\lambda \geq 0$, then $\SS\geq 0$ by the maximum principle and equation $(\ref{lapS})$.  Moreover, a complete GRS has positive scalar curvature unless it is isometric to the flat Euclidean space  \cite{prs09}.
\end{remark}

One main motivation of the study to GRS\rq{}s is that  they arise naturally as self-similar solutions to the Ricci flow. For a fixed GRS given by (\ref{GRS}) with $g(0)=g$ and $f(0)=f$, we define $\rho(t):=1-2\lambda{t}>0$,  and let $\phi(t):M^{n}\rightarrow M^{n}$ be a one-parameter family of diffeomorphisms generated by $X(t):=\frac{1}{\rho(t)}\nabla_{g(0)} f$. By pulling back, 
\begin{align*}
g(t)&=\rho(t) \phi(t)^{\ast} g(0),\\
\text{Rc}(t)&=\phi^{\ast}\Rc(0)=\frac{\lambda}{\rho(t)}g(t)-\He_{g(t)}{f(t)}.
\end{align*}
Then  $(M, g(t))$, $0\leq t<T$,  is a solution to the Ricci flow equation (\ref{RF}), where $T=\frac{1}{2\lambda}$ ($=\infty$) if $\lambda>0$ ($\lambda\leq 0$). Other important quantities along the flow are given below,
\begin{align*}
f(t)&=f(0)\circ \phi(t)=\phi(t)^{\ast}f,\\
\SS(t)&=\text{trace} (\Rc(t))=\frac{n\lambda}{\rho(t)}-\Delta_{g(t)}{f(t)},\\
f_{t}&=|\nabla{f}|_{g(t)}^2,\\
\tau(t)&=T-t=\frac{\rho(t)}{2\lambda},\\
u&=(4\pi\tau)^{-n/2}e^{-f},\\
\Psi(g,\tau,f)&=\int_{M}\Big(\tau(|\nabla{f}|^2+\SS)+f-n\Big)u d\mu\\
&= -\tau C(t)\int_{M} ud\mu.
\end{align*}

\subsection{Four-Manifolds}

In this subsection, we give a brief review of the algebraic structure of curvature and geometry on an oriented four-manifold $(M,g)$. \\

First we recall the Kurkani-Nomizhu product for $(2,0)$ symmetric tensors A and B, 
\[(A\circ B)_{ijkl}=A_{ik}B_{jl}+A_{jl}B_{ik}-A_{il}B_{jk}-A_{jk}B_{il}.\]
Then we have the following decomposition of curvature,
\begin{equation}\label{curvdecom}
\RR=\WW+\frac{S g\circ g }{2n(n-1)} + \frac{\EE \circ g}{n-2}=\WW-\frac{S g\circ g}{2(n-2)(n-1)} + \frac{\Rc \circ g}{n-2}.
\end{equation}   

In dimension four this becomes,
\begin{align*}
\RR &=\WW+\frac{\SS}{24} g\circ g+ \frac{1}{2}\EE \circ g=\WW+U+V, \\
|\RR|^2 &=|\WW|^2+|U|^2+|V|^2, \\
|U|^2 &=\frac{1}{2n(n-1)}\SS^2 =\frac{1}{24}\SS^2,\\
|V|^2 &=\frac{1}{n-2}|\EE|^2 =\frac{1}{2}|\EE|^2.
\end{align*}

An important feature in dimension four is that the Hodge star operator decomposes the space of bi-vectors ($\Lambda_{2}$) orthogonally according to the eigenvalues $\pm1$. The Riemannian curvature inherits this decomposition and, consequently, has a special structure. To be more precise, let $\{e_{i}\}_{i=1}^{4}$ be an orthonormal basis of the tangent space at any arbitrary point on M, then one pair of orthonormal bases of bi-vectors is given by 
\begin{align}
\label{Hodgebase}
\{\frac{1}{\sqrt{2}}(e_{12}+e_{34}),\frac{1}{\sqrt{2}}(e_{13}-e_{24}),\frac{1}{\sqrt{2}}(e_{14}+e_{23})\} &\mbox{ for } \Lambda_{2}^{+},\\
\{\frac{1}{\sqrt{2}}(e_{12}-e_{34}),\frac{1}{\sqrt{2}}(e_{13}+e_{24}),\frac{1}{\sqrt{2}}(e_{14}-e_{23})\} &\mbox{ for } \Lambda_{2}^{-}.\nonumber
\end{align}
Accordingly, the curvature now is 
\begin{equation}
\label{curdec}
\RR=
 \left( \begin{array}{cc}
A^{+} & C \\
C^{T} & A^{-} \end{array} \right),
\end{equation}
with C essentially the traceless part. It is easy to observe that $W(\Lambda_{2}^{\pm})\in \Lambda_{2}^{\pm}$, so it is unambiguous to define $W^{\pm}:=W^{|\Lambda_{\pm}}$. In particular,
\begin{equation}
\label{calpm}
\WW^{\pm}(\alpha, \beta)=\WW(\alpha^{\pm}, \beta^{\pm}),
\end{equation}
with $\alpha^{\pm}$ and $\beta^{\pm}$ the projection of $\alpha, \beta$ onto $\Lambda^{\pm}_{2}$.\\

Furthermore, as $\WW$ is traceless and satisfies the first Bianchi identity, there is a normal form by M. Berger \cite{berger61} (see also \cite{st69}). That is, there exists an orthonormal basis $\{e_{i}\}_{i=1}^4$ of $T_{p}M$, consequently $\{e_{12}, e_{13}, e_{14}, e_{34}, e_{42}, e_{23}\}$ being a basis of $\Lambda_{2}$, such that, 
\begin{equation}
\label{normalform}
\WW=
 \left( \begin{array}{cc}
A & B \\
B & A \end{array} \right)
\end{equation}
for $A=\text{diag}(a_{1}, a_{2}, a_{3})$ and $B=\text{diag}(b_{1}, b_{2}, b_{3})$, with $a_1+a_2+a_3=b_1+b_2+b_3=0$. Then, by (\ref{calpm}),
\[\WW^{\pm}=
 \left( \begin{array}{cc}
\frac{A\pm B}{2} & \frac{B\pm A}{2} \\
\frac{B\pm A}{2} & \frac{A\pm B}{2} \end{array} \right).\]

Using the basis given in (\ref{Hodgebase}),  we get
\[\WW=
 \left( \begin{array}{cc}
A+B & 0 \\
0 & A-B \end{array} \right).\]

Hence we obtain the following well-known identities.
\begin{lemma}
\label{sumW1}
Let $(M^4, g)$ be a four-dimensional Riemannian manifold, then the following tensor equations hold, 
\begin{align}
\label{sumW}
(\WW^{\pm})_{ikpq}(\WW^{\pm})_{j}{}^{kpq}&=|\WW^{\pm}|^2 g_{ij},\\
(\WW^{\pm})_{ikpq}(\WW^{\pm})^{kpq}{}_{j}&=\frac{1}{2}|\WW^{\pm}|^2 g_{ij}.
\end{align}

\end{lemma}
\begin{proof}
Note that these identities only depend on the decomposition of these tensors. In particular, it suffices to prove for the Weyl tensor. Using the normal form discussed above, we calculate that,
\begin{align*}
\WW_{1kpq}\WW^{kpq}{}_{1}&=\sum_{i=1}^3 a_{i}^2 -2(b_{1}b_{2}+b_{2}b_3 + b_{3}b_1)\\
&=\sum_{i=1}^3 (a_{i}^2+b_{i}^2). 
\end{align*}
Similar calculation can be done for all other pair of indexes to verify the statements. 
\end{proof}

\begin{remark}
The first identity can also be found in \cite[2.31]{Derd83}.
\end{remark}

In addition, in reference to the decomposition of curvature in (\ref{curdec}), we have the following relations.
\begin{align*}
A^{\pm} &=W^{\pm}+\frac{S}{12}I^{\pm},\\
|A^{\pm}|^{2} &=|W^{\pm}|^{2}+\frac{S^{2}}{48}, \\
|\EE|^2 &=|\text{Rc}|^2-\frac{S^2}{4}=4|C|^2=4\text{tr}(CC^T).
\end{align*}

If the manifold is closed, then the Gauss-Bonnet-Chern formula for the Euler characteristic and Hirzebauch formulas for the signature (cf. \cite{besse} for more details) are given by,
\begin{align}
\label{Euler}
8\pi^2 \chi(M) &= \int_{M} (|\WW|^2-|V|^2+|U|^2)=\int_{M}(|\WW|^2-\frac{1}{2}|\EE|^2+\frac{S^2}{24})\nonumber\\
& =\int_{M}(|\RR|^2-|\EE|^2),\\
\label{signature}
12 \pi^2 \tau(M) &=\int_{M}(|\WW^{+}|^2-|\WW^{-}|^2). 
\end{align} 
\begin{remark} It follows immediately that if M admits an Einsterin metric $\EE=0$, then we have the Hitchin-Thorpe inequality
\begin{equation*}
|\tau(M)|\leq \frac{2}{3}\chi(M).
\end{equation*}
\end{remark}

The Hodge operator in dimension four is related to a certain decomposition  on the tangent bundle. Let $\{\alpha_{i}\}_{i=1}^{3}$ be a positive-oriented orthogonal basis of $\Lambda_{2}^{+}$ with $|\alpha_{i}|=\sqrt{2}$, according to  \cite{ahs78}, if $\text{sign}(i,j,k)=1$, then we have  
\begin{align}
\label{baseplus}
\alpha_{i}^2 &= -\text{Identity},\nonumber \\
\alpha_{i}\alpha_{j} &=\alpha_{k}=-\alpha_{j}\alpha_{i}, \nonumber \\
\left\langle{\alpha_{i}(X),\alpha_{j}(X)}\right\rangle&=\left\langle{X,-\alpha_{i}\alpha_{j}X}\right\rangle =\left\langle{X, \alpha_{k}X}\right\rangle =0. \nonumber
\end{align}
Here $\text{sign}(i,j,k)$ is the sign-um of the permutation of $\{1,2,3\}$. The positive-orientation is just to agree with the sign convention. An example of such a basis is given by multiplying $\sqrt{2}$ the basis given in (\ref{Hodgebase}). Consequently, we have the following result.

\begin{lemma}\label{decom}
Suppose $(M,g)$ is a four-dimensional Riemannian manifold and $X$ is a vector field on M. At any point p such that $X_{p} \neq 0$, $$T_{p}M=X_{p}\oplus \Lambda_{2}^{+}(X_{p}),$$ 
with $\Lambda_{2}^{+}(X)=\{\alpha(X_{p}), \alpha \in \Lambda_{2}^{+}\}$. 
\end{lemma}
\begin{proof}
Pick an orthogonal basis of $\Lambda_{2}^{+}$ as above then it follows that $\{\alpha_{i}(X_{p})\}_{i=1}^3$ are three orthogonal vectors and each is perpendicular to $X_{p}$. So the statement follows. 
\end{proof}
\begin{remark} \label{derem}
By symmetry, the statement also holds for $\Lambda^{-}_{2}$. When the context is clear, we normally omit the sub-index of the point. 
\end{remark}

\subsection{New Sectional Curvature}
In this subsection, we first prove some results in dimension four to illustrate that classical techniques for Einstein 4-manifolds can be adapted to study  GRS's.\\

For a four-dimensional GRS $(M,g,f,\lambda)$, define
\begin{equation}
\label{equaH}
H=\He{f}\circ g, 
\end{equation}
then the following decomposition follows from a straightforward calculation.
\begin{lemma}
\label{calculateH}
With respect to the decomposition given by (\ref{Hodgebase}), we have
\[H=
 \left( \begin{array}{cc}
A & B \\
B^{T} & A\end{array} \right),\]
with 
\begin{align*}
A&=\frac{\Delta{f}}{2} \text{Id}~ ,
\\
B&=\left(\begin{array}{ ccc}
  \frac{f_{11}+f_{22}-f_{33}-f_{44}}{2} & f_{23}-f_{14} & f_{24}+f_{13}\\ 
  f_{23}+f_{14} & \frac{f_{11}+f_{33}-f_{22}-f_{44}}{2} & f_{34}-f_{12} \\ 
  f_{24}-f_{13} & f_{34}+f_{12} & \frac{f_{11}+f_{44}-f_{22}-f_{33}}{2} \end{array}  \right).
 \end{align*}
 \end{lemma}
\begin{remark} \label{HW0}
In particular $\left\langle{H,\WW}\right\rangle=0$.
\end{remark}

We further define a new ``curvature" tensor $\overline{\RR}$ by 
\begin{align}
\label{newcur}
\overline{\RR}&=\RR+\frac{1}{2}H\\
&=\WW+\frac{S}{24} g\circ g+ \frac{1}{2}(\Rc-\frac{S}{4}g) \circ g+\frac{1}{2}H \nonumber \\
&=\WW-\frac{S}{12} g\circ g +\frac{1}{2}\lambda g\circ g=\WW+(\frac{\lambda}{2}-\frac{S}{12})g\circ g.\nonumber
\end{align}
Thus, it follows immediately that, with respect to (\ref{Hodgebase}),
\[ \overline{\RR}=
 \left( \begin{array}{cc}
\overline{A}^{+} & 0 \\
0 & \overline{A}^{-} \end{array} \right),\]
with $\overline{A}^{\pm}=\WW^{\pm}+(\lambda-\frac{S}{6})\text{Id}=\WW^{\pm}+(\frac{\Delta{f}}{4}+\frac{S}{12})\text{Id}$. Furthermore, following the argument in \cite{berger61}, we obtain,

\begin{proposition}
\label{normalformG} There exists a normal form for $\overline{\RR}$. More precisely, at each point, there exits an orthonormal base $\{e_{i}\}_{i=1}^{4}$, such that with respect to the corresponding base $\{e_{12}, e_{13}, e_{14}, e_{34}, e_{42}, e_{23}\}$ for $\Lambda^2$ and as an operator on 2-forms, 
\[\overline{\RR}=
 \left( \begin{array}{cc}
A & B \\
B & A \end{array} \right),\]
with $A=\text{diag}(a_{1}, a_{2}, a_{3})$ and  $B=\text{diag}(b_{1},b_{2}, b_{3})$. Moreover, $a_{1}=\min \overline{K}$, $a_{3}=\max \overline{K}$ and $|b_{i}-b_{j}|\leq |a_{i}-a_{j}|$, where $\overline{K}$ is the \lq\lq{}sectional curvature\rq\rq{} of $\overline{\RR}$, i.e., $\overline{K}(e_{1},e_{2})=\overline{\RR}_{1212}$ for any orthonormal vectors $e_{1}$ and $e_{2}$. 
\end{proposition}

\begin{remark} Can a GRS be characterized by the existence of such a function f with  $\overline{\RR}$ constructed as above having the normal form?
\end{remark} 

Next, we investigate the assumption of having a lower bound on this new sectional curvature similar to \cite{gl99}. For $\epsilon<1/3$, suppose that 
\begin{equation}
\label{seclb}
\overline{K}\geq \epsilon \lambda.
\end{equation}
 Equivalently, for any orthonormal pairs $e_{i}$ and $e_{j}$, that is
\begin{equation}
\label{seclb1}
\overline{\RR}_{ijij}\geq \epsilon\lambda \Leftrightarrow \RR_{ijij}+\frac{f_{ii}+f_{jj}}{2}\geq \epsilon \lambda.
\end{equation}

Then we have the following lemma.
\begin{lemma} 
\label{lemma}
Let $(M,g,f, \lambda)$ be a GRS, then assumption (\ref{seclb}) implies the following:
\begin{align*}
\SS+3\Delta{f} &\geq 12 \epsilon \lambda,\\
\SS & \leq 6(1-\epsilon)\lambda,\\
\Delta{f} &\geq 2(3\epsilon-1)\lambda,\\
\frac{1}{\sqrt{6}}(|\WW^{+}|+|\WW^{-}|)&\leq 2(1-\epsilon)\lambda-\frac{S}{3}.
\end{align*}
The equality happens in the last formula if and only if $\WW^{\pm}$ has the form $a^{\pm}\text{diag}(-1,-1, 2)$, with $a^{\pm}\geq 0$ and 
\[a^{+}+a^{-}=2(1-\epsilon)\lambda-\frac{\SS}{3}.\]
\end{lemma}

\begin{proof}
All  inequalities follow from tracing equation (\ref{seclb1}) and the soliton equation $S+\Delta{f}=4\lambda$ except the last one. 

For the last inequality, first note that any two form $\phi$ can be written as a simple wedge product of 1-forms iff $\phi\wedge \phi=0$. In dimension four, with respect to (\ref{Hodgebase}), that is equivalent to $\phi=\phi^{+}+\phi^{-}$ and $|\phi^{+}|=|\phi^{-}$. Therefore, in light of Proposition \ref{normalform}, assumption (\ref{seclb}) is equivalent to 
\begin{equation}
\label{lemma1}
a^{+}+a^{-}+2\lambda-\frac{S}{3}\geq 2\epsilon \lambda 
\end{equation}
with $a^{+}, a^{-}$ are the smallest eigenvalues of $\WW^{\pm}$. Using the algebraic inequalities 
\begin{align}
-a^{+} &\geq \frac{1}{\sqrt{6}}|\WW^{+}|,\label{lemma2}\\
-a^{-} &\geq \frac{1}{\sqrt{6}}|\WW^{-}|, \label{lemma21}\end{align}
 we obtain:
\[2(1-\epsilon)\lambda-\frac{S}{3}\geq \frac{1}{\sqrt{6}}(|\WW^{+}|+|\WW^{-}|).\]
Equality happens if and only if the equality happens in (\ref{lemma1}) and (\ref{lemma2}) (or \eqref{lemma21}). The result then follows immediately.

\end{proof}

\begin{lemma}
\label{sumWnorm}
 Let $(M,g,f, \lambda)$ be a closed GRS with assumption (\ref{seclb}),  then
\[
\int_{M}(|\WW^{+}|+|\WW^{-}|)^2 \leq \int_{M}\frac{2S^2}{3}d\mu-8(1-\epsilon)(1+3\epsilon)\lambda^2 V(M).
\]
Again equality holds if $\WW^{\pm}$ has the form $a^{\pm}\text{diag}(-1,-1, 2)$ with $a^{\pm}\geq 0$ and 
\[a^{+}+a^{-}=2(1-\epsilon)\lambda-\frac{\SS}{3}.\]
\end{lemma}
\begin{proof}
Applying Lemma \ref{lemma}, we compute
\begin{align*}
\int_{M}(2(1-\epsilon)\lambda-\frac{S}{3})^2 =&4(1-\epsilon)^2 \lambda^{2} V(M)-\frac{4(1-\epsilon)\lambda}{3}\int_{M} S +\int_{M}\frac{S^2}{9} \\
=& 4(1-\epsilon)^2 \lambda^{2} V(M)-\frac{4(1-\epsilon)\lambda}{3}4\lambda V(M)+\int_{M}\frac{S^2}{9}\\
=&  4(1-\epsilon) \lambda^{2} V(M) (-\epsilon-\frac{1}{3})+\int_{M}\frac{S^2}{9} .
\end{align*}

\end{proof}
\begin{remark} If we use $S\leq 6(1-\epsilon)\lambda$, then  
\[\int_{M}(|\WW^{+}|+|\WW^{-}|)^2 \leq (\int_{M} S^2 d\mu)(\frac{2}{3}-\frac{2(1+3\epsilon)}{9(1-\epsilon)})=\frac{4(1-3\epsilon)}{9(1-\epsilon)}\int_{M} S^2  .\]
\end{remark}

\begin{lemma} 
\label{intRc}
Let $(M,g,f, \lambda)$ be a closed GRS, then
\[\int_{M}|\Rc|^2 =\int_{M} \frac{\SS^2}{2}  -4\lambda^2 V(M).\]
\end{lemma}

\begin{proof}
Using equation (\ref{lapS}), we compute:
\begin{align*}
2\int_{M}|\Rc|^2 d\mu &=\int_{M}(2\lambda \SS+\left\langle{\nabla{f},\nabla{\SS}}\right\rangle) d\mu\\
&=2\lambda 4\lambda V(M)-\int_{M}\Delta{f}\SS d\mu d\mu \\
&=8\lambda^{2} V(M)-\int_{M}(4\lambda-\SS)\SS d\mu\\
&=-8\lambda^{2} V(M)+\int_{M}\SS^2 d\mu.
\end{align*}

\end{proof}
The above  results lead to the following estimate on the Euler characteristic.

\begin{proposition}  Let $(M,g,f, \lambda)$ be a closed non-flat  GRS with unit volume, satisfying assumption (\ref{seclb}), then  
$$8\pi^2\chi(M)< \frac{7}{12}\int_{M}\SS^2 d\mu +2\lambda^2 (12\epsilon^2-8\epsilon-3).$$
\end{proposition}

\begin{proof}
By the Gauss-Bonnet-Chern formula,
\begin{align*}
8\pi^2\chi(M)=&\int_{M}(|\WW|^2-\frac{1}{2}|\EE|^2+\frac{\SS^2}{24})d\mu\\
\leq &\int_{M}(|\WW^{+}|+|\WW^{-}|)^2 d\mu+\frac{1}{2}\int_{M}|\Rc|^2 d\mu-\int_{M}\frac{\SS^2}{12}d\mu.
\end{align*}

Applying Lemmas \ref{sumWnorm} and  \ref{intRc} yields the inequality. \\

We now claim that the equality case can not happen. Suppose otherwise then $|\WW^{+}||\WW^{-}|=0$ and equality also happens in Lemma \ref{sumWnorm}. By the regularity theory for solitons \cite{bando87}, we can choose an orientation such that $|\WW^{-}|\equiv 0$. Hence $W^{+}=\text{diag}(-a^{+}, -a^{+}, 2a^{+})$ with $a^{+}=2(1-\epsilon)\lambda-\frac{S}{3}$, then by \cite[Theorem 1.1]{cw11}, we have $W^{+}=0$ or $\Rc=0$ .

 In the first case,  by the classification of locally conformally flat four-dimensional closed GRS's as discussed in Introduction, $(M,g)$ is flat, this is a contradiction.
 
 In the second case, $\Rc=0$ implies $\SS=0=\lambda$, and since equality happens in Lemma \ref{sumWnorm}, $\WW^{+}=0$. Hence the above argument applies. 

\end{proof}
\begin{remark} The Euler characteristic of a closed Ricci soliton has been studied by \cite{Derd06}. If the manifold is Einstein and $\epsilon=0$, we recover some results of \cite{gl99}. 
\end{remark}

\section{\textbf{A Bochner-Weitzenb\"ock Formula}}

In this section, we prove Theorem \ref{WSoliton}, a new Bochner-Weitzenb\"ock formula for the Weyl tensor of GRS's, which generalizes the one for Einstein manifolds. Bochner-Weitzenb\"ock formulas have been proven a powerful tool to find connections between topology and geometry with certain curvature conditions (for example, see \cite{g98, pe06book, wu88}). 

   

Particularly, in dimension four, if $\delta \WW^{+}=0$ (this contains all Einstein manifolds), we have the following well-known formula (see \cite[16.73]{besse}),
\begin{equation}
\label{bwdiv}
\Delta|\WW^{+}|^2=2|\nabla \WW^{+}|^2+\SS|\WW^{+}|^2-36 \text{det} \WW^{+}.
\end{equation}
This equation plays a crucial role to obtain a $L_{2}$-gap theorem of the Weyl tensor and to study the classification problem of Einstein manifolds (cf. \cite{g00, gl99,yangdg00}). \\


Our first technical lemma gives a formula of $\Delta_{f}W$ in a local frame. Also it is noticed that the Einstein summation convention is used repeatedly here. 

\begin{lemma} 
\label{generalw}
Let $(M,g,f, \lambda)$ be a GRS and $\{e_{i}\}_{i=1}^{n}$ be a local normal frame, then the following holds,
\begin{align}
\Delta_{f}W_{ijkl}=& 2\lambda W_{ijkl}-2(C_{ijkl}-C_{ijlk}+C_{ikjl}-C_{iljk}) \nonumber \\
&-\frac{2}{(n-2)^2}g^{pq}(\Rc_{ip}\Rc_{qk}g_{jl}-\Rc_{ip}\Rc_{ql}g_{jk}+\Rc_{jp}\Rc_{ql}g_{ik}-\Rc_{jp}\Rc_{qk}g_{il}) \nonumber\\
&+\frac{2\SS}{(n-2)^2}(\Rc_{ik}g_{jl}-\Rc_{il}g_{jk}+\Rc_{jl}g_{ik}-\Rc_{jk}g_{il}) \\
&-\frac{2}{n-2}(\RR_{ik}\RR_{jl}-\RR_{jk}\RR_{il})-\frac{2(\SS^2-|\Rc|^2)}{(n-1)(n-2)^2}(g_{ik}g_{jl}-g_{il}g_{jk}),\nonumber
\end{align}
here $C_{ijkl}=g^{pq}g^{rs}\WW_{pijr}\WW_{slkq}$.
\end{lemma}
\begin{proof}
First, we recall how a GRS can be realized as a self-similar solution to the Ricci flow (\ref{RF}), as in Section 2.1. 

Let $\tau(t)=1-2\lambda t$ and $\phi(x,t)$ is a family of diffeomorphisms generated by $X=\frac{1}{\tau}\nabla{f}$. For $g(0)=g, g(t)=\tau(t)\phi^{*}(t)g$, then $(M,g(t))$ is a solution to the Ricci flow. Furthermore, $\WW(t)=\tau\phi^{*}\WW$.
Let $p$ be  a point  in M and $\{e_{i}\}_{i=1}^n$ be a basis of $T_{p}M$, and we obtain a local normal frame via  
extending $e_{i}$ to a neighborhood by parallel translation along geodesics with respect to $g(0)$. 
First, we observe,
\begin{equation}\label{w1} 
\ddt \WW(t)_{ijkl}\mid_{t=0}= (\ddt \tau \phi^{*}\WW)_{ijkl}\mid_{t=0}
=-\frac{2\lambda}{\tau}\WW_{ijkl}+(L_{\nabla{f}}\WW)_{ijkl},
\end{equation}
where $L_{X}$ is the Lie derivative with respect to X. Furthermore, by definition, 
\begin{align}
L_{\nabla f}\WW_{ijkl}=&\nabla{f}(\WW_{ijkl})-\WW([\nabla{f},e_{i}],e_{j},e_{k},e_{l})-\WW(e_{i},[\nabla{f},e_{j}],,e_{k},e_{l})\nonumber \\
\label{w2}
&-\WW(e_{i},e_{j}, [\nabla{f},e_{k}],e_{l})-\WW(e_{i},e_{j},e_{k},[\nabla{f},e_{l}]).
\end{align} We calculate that 
\begin{align}
\WW([\nabla{f},e_{i}],e_{j},e_{k},e_{l})=& \WW(\nabla_{\nabla f}e_{i}-\nabla_{e_{i}}\nabla{f}, e_{j},e_{k},e_{l})=-\WW(\nabla_{e_{i}}\nabla{f},e_{j},e_{k},e_{l}).\nonumber
\end{align}
By the soliton structure, $\nabla_{e_{i}}\nabla_{.}f =-\Rc(e_{i},.)+\lambda g(e_{i},.) $. Thus,
\begin{align}
\WW([\nabla{f},e_{i}],e_{j},e_{k},e_{l})
&=-\WW(\lambda e_i-\Rc(e_i), e_j, e_k, e_l)\nonumber \\
\label{w3} &=-\lambda\WW_{ijkl}+g^{pq}\Rc_{ip}\WW_{qjkl}.
\end{align}
 Combining (\ref{w1}),(\ref{w2}), and 
(\ref{w3}) we obtain,
\begin{align*}
\ddt \WW(t)_{ijkl}\mid_{t=0}=&\nabla{f}(\WW_{ijkl})+2\lambda \WW_{ijkl}\\
&-g^{pq}(\Rc_{ip}\WW_{qjkl}+\Rc_{jp}\WW_{iqkl}+\Rc_{kp}\WW_{ijql}+\Rc_{ip}\WW_{qjkl}).
\end{align*}
Along the Ricci flow, the Weyl tensor is evolving under the following  (for example, see \cite[Prop 1.1]{cm11})
\begin{align}
\ddt \WW(t)_{ijkl}\mid_{t=0}=&\Delta(\WW_{ijkl})+2(C_{ijkl}-C_{ijlk}+C_{ikjl}-C_{iljk})\nonumber \\
&-g^{pq}(\Rc_{ip}\WW_{qjkl}+\Rc_{jp}\WW_{iqkl}+\Rc_{kp}\WW_{ijql}+\Rc_{ip}\WW_{qjkl})\nonumber \\
&+\frac{2}{(n-2)^2}g^{pq}(\Rc_{ip}\Rc_{qk}g_{jl}-\Rc_{ip}\Rc_{ql}g_{jk}+\Rc_{jp}\Rc_{ql}g_{ik}-\Rc_{jp}\Rc_{qk}g_{il})\nonumber\\
\label{generalW1}
&+\frac{2\SS}{(n-2)^2}(\Rc_{ik}g_{jl}-\Rc_{il}g_{jk}+\Rc_{jl}g_{ik}-\Rc_{jk}g_{il}) \\
&+\frac{2}{n-2}(\RR_{ik}\RR_{jl}-\RR_{jk}\RR_{il})+\frac{2(S^2-|\Rc|^2)}{(n-1)(n-2)^2}(g_{ik}g_{jl}-g_{il}g_{jk}).\nonumber
\end{align}
The result then follows. 
\end{proof}

Furthermore, in dimension four, we are able to obtain significant  simplification due to the special structure given by the Hodge operator. That gives the proof of our first main theorem.

\begin{proof} {\bf (Theorem \ref{WSoliton})}
We observe that, 
\begin{align*}
\left\langle{\WW^{+},\Delta_{f}\WW^{+}}\right\rangle&=\left\langle{\WW^{+},\Delta\WW^{+}}\right\rangle-\left\langle{\WW^{+},\nabla_{\nabla f}\WW^{+}}\right\rangle\\
&=\left\langle{\WW^{+},\Delta\WW^{+}}\right\rangle-\frac{1}{2}\nabla_{\nabla f}|\WW^{+}|^2.
\end{align*}
Therefore,
\[\Delta_{f} |\WW^{+}|^2=\Delta |\WW^{+}|^2-\nabla_{\nabla f}|\WW^{+}|^2=2\left\langle{\WW^{+},\Delta_{f} \WW^{+}}\right\rangle+2|\nabla \WW^{+}|^2.\]
To calculate the first term of the right hand side, we use the normal form of the Weyl tensor 
(\ref{normalform}). As usual, a local normal frame is obtained by parallel translation along geodesic lines. Then (\ref{Hodgebase}) gives a basis of eigenvectors $\{\alpha_{i} \}_{i=1}^3$ of $\WW^{+}$ with corresponding eigenvalues $\lambda_{i}=a_{i}+b_{i}$. Consequently, 
\begin{equation}
\label{mt11}
\left\langle{\WW^{+},\Delta_{f} \WW^{+}}\right\rangle=\sum_{i}\lambda_{i}\Delta_{f} \WW^{+}(\alpha_{i},\alpha_{i}).
\end{equation}
In order to use Lemma \ref{generalw}, it is necessary to calculate the $C_{ijkl}$ terms. By the normal form, we have
\begin{align*}
C_{1212}&=a_1^2+b_{2}^2+b_{3}^2, &C_{1234}&=-2a_1b_3,\\
C_{1221}&=-2b_2b_3, &C_{1243}&=2a_1b_2,\\
C_{1122}&=2a_2a_{3}, &C_{1324}&=2a_2b_3,\\
C_{1221}&=-2b_2b_3, &C_{1423}&=-2a_3b_2.
\end{align*}
Thus, 
\begin{align*}
\Delta_{f}\WW_{1212}=&2\lambda a_{1}-2(a_{1}^2+b_{1}^2+2a_{2}a_{3}+2b_{2}b_{3})\\
&-\frac{1}{2}\sum_{p}(\Rc_{1p}^2+\Rc_{2p}^2)+\frac{S}{2}(\Rc_{11}+\Rc_{12})\\
&-(\Rc_{11}\RR_{22}-\Rc_{12}^2)-\frac{1}{6}(S^2-|\Rc|^2),\\
\Delta_{f}\WW_{1234}=&2\lambda b_{1}-4(a_{1}b_{1}+a_{2}b_{3}+a_{3}b_{2})+(\Rc_{13}\Rc_{24}-\Rc_{23}\Rc_{14}).
\end{align*}
Therefore,
\begin{equation}
\label{mt12}
\Delta_{f}\WW^{+}(\alpha_{1},\alpha_{1})=2\lambda \lambda_{1}-2\lambda_{1}^2-4\lambda_{2}\lambda_{3}-\frac{1}{12}(|\Rc|^2-S^2)-T_{1},
\end{equation}
in which,
\begin{align*}
2T_{1}=&\Rc_{11}\Rc_{22}+\Rc_{33}\Rc_{44}+2\Rc_{13}\Rc_{24}-\Rc_{12}^2-2\Rc_{23}\Rc_{14}-\Rc_{34}^2\nonumber\\
=& (\Rc\circ \Rc)(\alpha_{1},\alpha_{1}).\nonumber
\end{align*}
Similar calculations hold when replacing $\alpha_1$ by $\alpha_{2}$, $\alpha_{3}$,
\begin{align}
\label{mt13}
\Delta_{f}\WW^{+}(\alpha_{2},\alpha_{2})=&2\lambda \lambda_{2}-2\lambda_{2}^2-4\lambda_{1}\lambda_{3}-\frac{1}{12}(|\Rc|^2-S^2)-\frac{1}{2}\Rc\circ  \Rc (\alpha_2, \alpha_2),\\
\label{mt14}
\Delta_{f}\WW^{+}(\alpha_{3},\alpha_{3})=&2\lambda \lambda_{3}-2\lambda_{3}^2-4\lambda_{1}\lambda_{2}-\frac{1}{12}(|\Rc|^2-S^2)-\frac{1}{2}\Rc\circ  \Rc (\alpha_3, \alpha_3).
\end{align}
 Combining (\ref{mt11}), (\ref{mt12}), (\ref{mt13}), (\ref{mt14}) yields,
\begin{align*}
\left\langle{\WW^{+},\Delta_{f}\WW^{+}}\right\rangle=&2\lambda |\WW^{+}|^2-18\text{det}\WW^{+}-\sum_{i}T_{i}\lambda_{i}\\
=&2\lambda |\WW^{+}|^2-18\text{det}\WW^{+}-\frac{1}{2}\left\langle{\Rc\circ \Rc,\WW^{+}}\right\rangle. 
\end{align*}
The first equality then follows. The second equality comes from the soliton equation, the  property that $\WW^{+}$ is trace-free and Remark \ref{HW0}. 
\end{proof}

\section{\textbf{Applications of the Bochner-Weitzenb\"ock Formula}}

\subsection{A Gap Theorem for the Weyl Tensor}
 In \cite{g00}, under the assumptions $\WW^{+}\neq 0$, $\delta\WW^{+}=0$, and the positivity of the Yamabe constant, M. Gursky proves the following inequality, relating $||\WW^{+}||_{L_{2}}$ with topological invariants of a closed four-manifold,
 \begin{equation}
 \label{WEinstein}
 \int_{M}|\WW^{+}|^2 d\mu\geq \frac{4}{3}\pi^2(2\chi(M)+3\tau(M)).
 \end{equation}
 Our main result in this section  is to prove an analog for GRS's. It is noted that the particular structure of GRS allows us to relax the harmonic self-dual condition above  with the cost of a worse coefficient due to the absence of an improved Kato\rq{}s inequality.
 
 \begin{theorem}
 \label{gapW}
  Let $(M,g,f, \lambda)$ be a closed four-dimensional shrinking GRS with
  \begin{equation}
  \label{gapWa}
  \int_{M}\left\langle{\WW^{+},\text{Hess}f\circ \text{Hess}{f}}\right\rangle \leq \frac{2}{3} \int \SS|\WW^{+}|^2,
  \end{equation}
  then, unless $\WW^{+}\equiv 0$, 
  \begin{equation}
  \label{gapW1}
  \int_{M}|\WW^{+}|^2 d\mu> \frac{4}{11}\pi^2(2\chi(M)+3\tau(M)).
  \end{equation}
  \end{theorem}  
 
  \begin{remark} 
  By Remark \ref{intanddivplus}, assumption (\ref{gapWa}) is equivalent to 
  \[\int |\delta \WW^{+}|^2\leq \int \frac{\SS}{6}|\WW^{+}|^2.\]
  Thus, it is clearly weaker than the assumption of harmonic self-dual.
  \end{remark}
  
 To prove Theorem \ref{gapW}, we follow an idea of \cite{g00} and introduce a Yamabe-type conformal invariant. 
First, the conformal Laplacian is given by, 
\begin{equation*}
L=-6\Delta +\SS.
\end{equation*}
Furthermore, we define that
\begin{align*}
F_{a, b}=&a\SS-b|\WW^{+}|,\\
L_{a, b}=&-6a\Delta_{g}+F_{a,b}=aL-b\WW^{+},
\end{align*} where $a$ and $b$ are constants to be determined later. 
Under a conformal transformation as described in (\ref{conf}), for any function $\Phi$, we have 
\begin{align*}
\tilde{L}(\Phi)=&u^{-3}L(\Phi u),\\
\widetilde{L}_{a,b}\Phi=& u^{-3}L_{a,b}(\Phi u),\\
\widetilde{F}_{a,b}=& u^{-3}(-6a\Delta_{g}+F_{a,b})u,\\
\int_{M} \widetilde{F}_{a,b} d\widetilde{\mu}=& \int_{M} u(-6a\Delta_{g}+F_{a,b})u d\mu \\
=&\int_{M}(F_{a,b} u^2+6a|\nabla u|^2)d\mu.
 \end{align*}

The Yamabe problem is, for a given Riemannian manifold $(M, g)$, to find a   constant scalar curvature metric in its conformal class $[g]$. That is equivalent to find a critical point of the following functional, for any $C^{2}$ positive function u, let $\tilde{g}=u^2 g$, define
\begin{equation*}
Y_{g}[u]=\frac{\left\langle{u, Lu}\right\rangle_{L_2}}{||u||_{L_{4}}^2}=\frac{\int_{M}\widetilde{\SS}d\widetilde{\mu}}{\sqrt{\int_{M}d\widetilde{\mu}}}.
\end{equation*}
Then the conformal invariant $Y$ is defined as 
\begin{equation*}
Y(M,[g])=\inf\{Y_{g}[u]\text{: u is a positive $C^{2}$ function on M}\}.
\end{equation*}
For an expository account on the Yamabe problem, see \cite{lp87}. \\

As $F_{a,b}$ conformally transforms like the scalar curvature, in analogy with the discussion above, we can define the following conformal invariant. 

\begin{definition} Given a Riemannian manifold $(M,g)$, define
\begin{equation*} 
\hat{Y}_{a,b}(M,[g])=\inf \{(Y_{a,b})_{g}[u]\text{: u is a positive $C^{2}$ function on M}\},
\end{equation*} where
\begin{equation*}
(\hat{Y}_{a,b})_{g}[u]=\frac{\left\langle{u, L_{a,b}u}\right\rangle_{L_2}}{||u||_{L_4}^2}=\frac{\int_{M}\widetilde{F_{a,b}}d\widetilde{\mu}}{\sqrt{\int_{M}d\widetilde{\mu}}}.
\end{equation*}
\end{definition}

For the case of interest, we shall denote
\begin{align*}
F=&F_{1, 6\sqrt{6}}=\SS-6\sqrt{6}|\WW^{+}|,\\
\hat{Y}(M)=&\hat{Y}_{1, 6\sqrt{6}}(M,[g]),
\end{align*}
when the context is clear. First we observe the following simple inequality. 

\begin{lemma}
\label{modelS}
Let $(M^n,g)$ be a closed n-dimensional Riemannian manifold which is not locally conformally flat, and $(S^n,g_{\text{sd}})$ be the sphere with standard metric. Then
\begin{equation}
\hat{Y}(M,[g])\leq Y(M,[g])<Y(S^n, [g_{\text{sd}}])=\hat{Y}(S^n, [g_{\text{sd}}]).
\end{equation}
\end{lemma}
\begin{proof}
The first inequality follows from the definition and the following observation. Given a metric g, a positive function u and $b\geq 0$, then
\[\left\langle{u, Lu}\right\rangle_{L_{2}}-\left\langle{u, L_{1,b}u}\right\rangle_{L_{2}}=\int_{M}b|\WW^{+}|u^2 d\mu\geq 0.
\]
The second inequality is a result of T. Aubin \cite{aubin76} and R. Schoen \cite{schoen84}. The last inequality is an immediate consequence of the fact that the standard metric on $S^n$ is locally conformally flat ($\WW=0$). \end{proof}

On a complete gradient shrinking soliton, the scalar curvature is positive unless the soliton is isometric to the flat Euclidean space \cite{prs09}. Therefore, if the GRS is not flat then the existence of a solution to the Yamabe problem \cite{lp87} implies that $Y_{g}>0$. This observation is essential because of the following result.

\begin{proposition} 
\label{gursky1}
 Let $(M,g)$ be a closed four-dimensional Riemannian manifold. If $Y(M)>0$ and $\hat{Y}(M)\leq 0$, then there is a smooth metric $\tilde{g}=u^2 g$ such that 
\begin{equation}
\label{guu2}
\int _{M}\widetilde{\SS}^2 d\widetilde{\mu}\leq 216 \int_{M} |\widetilde{\WW}^{+}|^2 d\widetilde{\mu}. 
\end{equation}
Furthermore,  the equality holds  only if $\hat{Y}(M)=0$ and $\widetilde{\SS}=6\sqrt{6}|\widetilde{\WW}|$ .
\end{proposition}
\begin{proof}
The proof is almost identical to \cite[Prop 3.5]{g00}. Thus, we provide a brief argument here. 
Through a conformal transformation, the Yamabe problem can be solved via variational approach for an appropriate eigenvalue PDE problem. In particular, the existence of  solution under the assumption $Y(M)<Y(S^n)$ depends solely on the analysis of  regularity of the Laplacian operator (but not on the reaction term) \cite[Theorem 4.5]{lp87}. 
 
In our case, $F$ conformally transforms as scalar curvature and Lemma \ref{modelS} holds, then there exists a minimizer $v$ for $\hat{Y}_{g}[.]$, such that under normalization $||v||_{L_4}=1$, the metric $\tilde{g}=v^2 g$ satisfies $\widetilde{F}=\widetilde{\SS}-6\sqrt{6}|\widetilde{\WW}^{+}|=\hat{Y}(M).$ Applying $Y(M)>0$ and $\hat{Y}(M)\leq 0$ we obtain,
\begin{align*}
\int_{M}\widetilde{\SS}^{2} d\tilde{\mu}&=\int_{M}6\sqrt{6}|\widetilde{\WW}^{+}|\widetilde{\SS}d\tilde{\mu}+\hat{Y}(M)\int_{M}\widetilde{\SS}d\tilde{\mu}\\
&\leq \int_{M}6\sqrt{6}|\widetilde{\WW}^{+}|\widetilde{\SS}d\tilde{\mu}\\
&\leq 6\sqrt{6}(\int_{M}|\widetilde{\WW}^{+}|^2 d\tilde{\mu})^{1/2} (\int_{M}|\widetilde{\SS}|^2 d\tilde{\mu})^{1/2}.
\end{align*}
Therefore, $\int_{M}\widetilde{\SS}^{2} d\tilde{\mu}\leq 216 \int_{M}|\widetilde{\WW}^{+}|^2 d\tilde{\mu}$. The equality case is attained if only if  $\tilde{g}$ attains the infimum,  $\hat{Y}(M)=0$ and $\widetilde{\SS}=6\sqrt{6}|\widetilde{\WW}|$. 
\end{proof}

\begin{proposition} 
\label{neghatY}
Let $(M, g,f,\lambda)$ be a closed four-dimensional shrinking GRS satisfying (\ref{gapWa}) and $\WW^{+}\neq 0$, then $\hat{Y}(M)\leq 0$. Moreover,  equality holds only if   $\WW^{+}$ has the form $\omega\text{diag}(-1,-1, 2)$ for some $\omega\geq 0$ at each point.
\end{proposition}

\begin{proof}
By Theorem \ref{WSoliton}, we have
\[\Delta_{f} |\WW^{+}|^2=2|\nabla \WW^{+}|^2+4\lambda |\WW^{+}|^2-36\text{det}_{\Lambda^2_{+}}\WW^{+} -\left\langle{\Rc\circ \Rc,\WW^{+}}\right\rangle.\]
Integrating both sides and applying  (\ref{gapWa}) yield
\[
\int_{M}\Delta_{f} |\WW^{+}|^2 d\mu \geq \int_{M}\Big[2|\nabla \WW^{+}|^2+(\frac{\SS}{3}+\Delta {f}) |\WW^{+}|^2-36\text{det}_{\Lambda^2_{+}}\WW^{+}\Big].
\]
Via integration by parts, we have
\[\int_{M} \nabla{f} (|\WW^{+}|^2) d\mu=\int_{M}\left\langle{\nabla{f}, \nabla{|\WW^{+}|^2}}\right\rangle d\mu=-\int_{M} \Delta{f}|\WW^{+}|^2 d\mu.\]
Therefore, we arrive at 
\[0\geq \int_M \Big(2|\nabla \WW^{+}|^2+\frac{\SS}{3}|\WW^{+}|^2-36\text{det}_{\Lambda^2_{+}}\WW^{+}\Big).\]
We also have the following pointwise estimates, 
\begin{align*}
|\nabla \WW^{+}|^2 &\geq |\nabla |\WW^{+}||^2,\\
-18\text{det}\WW^{+}&\geq -\sqrt{6}|\WW^{+}|^3.
\end{align*}
The first one is the classical Kato\rq{}s inequality while the second one is purely algebraic.
Thus, for $u=|\WW^{+}|$,  
\[ \int_{M} (\frac{1}{3}Fu^2+2|\nabla u|^2) d\mu \leq 0. \]
Hence if $|\widetilde{\WW^{+}}|>0$ everywhere then the statement follows. \\

If $|\widetilde{\WW^{+}}|=0$ somewhere, let $M_{\epsilon}$ be the set of points at which $|\widetilde{\WW^{+}}|<\epsilon$. By the analyticity of a closed GRS \cite{bando87}, $\text{Vol}(M_{\epsilon})\rightarrow 0$ as $\epsilon\rightarrow 0$. Let $\eta_{\epsilon}:[0,\infty)\rightarrow [0, \infty)$ be a $C^{2}$ positive function which is $\epsilon/2$ on $[0,\epsilon/2]$, identity on $[\epsilon, \infty)$ and $0\leq \eta_{\epsilon}\rq{}\leq 10$. If $u_{\epsilon}=\eta_{\epsilon} \circ u$, then $u_{\epsilon}$ is $C^{2}$ and positive. In addition, we have,
\begin{align*}
\int_{M}Fu_{\epsilon}^2 d\mu &\leq \int_{M-M_{\epsilon}}Fu^2 d\mu+C\epsilon^2 \text{Vol}(M_{\epsilon}),\\
\int_{M}|\nabla u_{\epsilon}|^2 d\mu &=\int_{M} |\eta_{\epsilon}\rq{} \nabla u|^2 d\mu \leq \int_{M-M_{\epsilon}}|\nabla u|^2 d\mu+C \text{Vol}(M_{\epsilon}),
\end{align*} 
where C is a constant depending on the metric. Therefore,  we have,
\[\inf_{\epsilon>0} \{\int_{M}(Fu_{\epsilon}^2+6|\nabla u_{\epsilon}|^2) d\mu\}\leq 0.\]  
Consequently, $\hat{Y}(M)\leq 0$. \\

Now, equality holds only if $\int_{M} (\frac{1}{3}Fu^2+2|\nabla u|^2) d\mu = 0$ and the equality happens in each point-wise estimate above. The result then follows.


\end{proof}

We are now ready to prove the main result of this section.
\begin{proof} ({\bf Theorem \ref{gapW}})

By Proposition \ref{neghatY}, we have $\hat{Y}(g)\leq 0$ and $Y(M)>0$. Otherwise $\SS=0$ and the GRS is flat by \cite{prs09}, which is a contradiction to $\WW^{+}\neq 0$.  
Therefore, following Proposition \ref{gursky1},  there is a conformal transformation $\tilde{g}=u^2 g$ with 
\begin{equation}
\label{gapW2}
\int_{M} \widetilde{\SS}^2 d\widetilde{\mu}\leq 216\int_{M} |\widetilde{\WW}^{+}|^2 d\widetilde{\mu}.
\end{equation}
According to (\ref{Euler}) and (\ref{signature}), 
\begin{align}
2\pi^2(2\chi(M)+3\tau(M))=&\int_{M}|\widetilde{\WW}^{+}|^2 d\widetilde{\mu}-\frac{1}{4}\int_{M}|\widetilde{\EE}|^2 d\widetilde{\mu}+\frac{1}{48}\int_{M}\widetilde{\SS}^2 d\widetilde{\mu}\nonumber \\
\label{gapW3}
&\leq \int_{M}|\widetilde{\WW}^{+}|^2 d\widetilde{\mu}+\frac{1}{48}\int_{M}\widetilde{\SS}^2 d\widetilde{\mu}\\
&\leq (1+\frac{9}{2})\int_{M}|\widetilde{\WW}^{+}|^2 d\widetilde{\mu}\nonumber.
\end{align}
Here we used (\ref{gapW2}) in the last step. Since $||\WW^{+}||_{L_2}$ is conformally invariant, (\ref{gapW1}) then follows.\\ 

Now the equality holds only if all equalities hold in (\ref{gapW3}), (\ref{gapW2}) and (\ref{gapWa}). The first one implies that $\tilde{g}$ is Einstein. Therefore, by \cite[Theorem 1]{gl99}, inequality (\ref{gapW2}) is strict unless $\SS\equiv 0$. But this is a contradiction to  $Y(M)>0$. Thus the inequality is strict.

\end{proof}

\subsection{Isotropic Curvature}

Another application is the following inequality which is an improvement of \cite[Prop 2.6]{xu12}.
\begin{proposition} Let $(M, g, f, \lambda)$ be a four-dimensional GRS, then we have
\begin{equation}
\Delta_{f}u\leq (2\lambda+\frac{3}{2}u-\SS)u-\frac{1}{4}|\Rc|^2
\end{equation}
in the distribution sense where $u(x)$ is the smallest eigenvalue of $\frac{S}{3}-2\WW_{\pm}$.
\end{proposition}
\begin{proof}
Let $X_{1234}=\frac{S}{3}-2\WW(e_{12}+e_{34},e_{12}+e_{34})$ for any 4--orthonormal basis.
We  use the normal form discussed in (\ref{normalform}) and obtain a local frame by parallel translation along geodesic lines. 
We denote $\{\alpha_{i}\}_{i=1}^{3}$ the basis of $\Lambda_{2}^{+}$ as in (\ref{Hodgebase}) with corresponding eigenvalues $\lambda_{i}=a_{i}+b_{i}$. Without loss of generality, we can assume  $a_{1}+b_{1}\geq a_{2}+b_{2}\geq a_{3}+b_{3}$ and thus  $u(x)=X_{1234}(x)$. Using Lemma \ref{generalw}, we compute
\begin{align*}
\Delta_{f}\WW_{1212}=&2\lambda a_{1}-2(a_{1}^2+b_{1}^2+2a_{2}a_{3}+2b_{2}b_{3})\\
&-\frac{1}{2}\sum_{p}(\Rc_{1p}^2+\Rc_{2p}^2)+\frac{S}{2}(\Rc_{11}+\Rc_{12})\\
&-(\Rc_{11}\RR_{22}-\Rc_{12}^2)-\frac{1}{6}(S^2-|\Rc|^2),\\
\Delta_{f}\WW_{1234}=&2\lambda b_{1}-4(a_{1}b_{1}+a_{2}b_{3}+a_{3}b_{2})+(\Rc_{13}\Rc_{24}-\Rc_{23}\Rc_{14}).
\end{align*}
Let us recall that, $\Delta_{f}\SS=2\lambda \SS-2|\Rc|^2$. Thus, for $2T_{1}= (\Rc\circ \Rc)(\alpha_{1},\alpha_{1})$, we have 
\begin{align*}
\Delta_{f}(X_{1234})=&2\lambda \frac{\SS}{3}-\frac{2}{3}|\Rc|^2-4\lambda (a_{1}+b_{1})+4\lambda_{1}^2+8\lambda_{2}\lambda_{3}+\frac{1}{6}(|\Rc|^2-\SS^2)+T_{1} \\
=&2\lambda X_{1234}-\frac{1}{2}|\Rc|^2+4\lambda_{1}^2+8\lambda_{2}\lambda_{3}-\frac{1}{6}\SS^2+T_{1}.
\end{align*}
Next we observe that $\lambda_{2}+\lambda_{3}=-\lambda_{1}$ and $8\lambda_{2}\lambda_{3}\leq 2\lambda_{1}^2$. By Cauchy-Schwartz inequality, $T_{1}\leq \frac{1}{4}|\Rc|^2$. Therefore,
\begin{align*}
\Delta_{f}(X_{1234})\leq &2\lambda X_{1234}-\frac{1}{4}|\Rc|^2+6(\frac{\frac{\SS}{3}-X_{1234}}{2})^2-\frac{1}{6}\SS^2\\
\leq &2\lambda X_{1234}+\frac{3}{2}X_{1234}^2-\SS X_{1234}-\frac{1}{4}|\Rc|^2=u(2\lambda+\frac{3}{2}u-\SS)-\frac{1}{4}|\Rc|^2. 
\end{align*}
Since $\Delta_{f}u\leq \Delta_f (X_{1234})$ in the barrier sense of E. Calabi (see\cite{calabi58}), the result then follows.
\end{proof}

\section{\textbf{A Framework Approach}}

In this section, we shall propose a framework to study interactions between components of curvature operator and the potential function on a GRS $(M,g,f,\lambda)$. In particular, we represent the divergence and the interior product $i_{\nabla{f}}$ on each curvature component as linear combinations of four operators $P, Q, M, N$. The geometry of these operators, in turn, gives us information about the original objects. It should be noted that some identities here have already appeared elsewhere.\\

Let $(M^{n},g)$ be an n-dimensional oriented Riemannian manifold. Using the point-wise induced inner product, any anti-symmetric (2,0) tensor $\alpha$ (a two-form) can be seen as an operator on the tangent space by
\begin{equation*}
\alpha(X,Y)=\left\langle{-\alpha(X), Y}\right\rangle=\left\langle{X,\alpha(Y)}\right\rangle=
\left\langle{\alpha, X\wedge Y}\right\rangle.
\end{equation*}
In particular, a bi-vector acts on a vector $X$ as follows
\[(U\wedge V)X=\left\langle{V,X}\right\rangle U-\left\langle{U,X}\right\rangle V.\]

For instance, in dimension four,  the complete description is given by the table below.
\begin{equation}
\label{bivectoract}
\begin{tabular}
{ | l | c | c | c | c | c | r |}
  \hline                       
  \text{} & $e_{12}+e_{34}$ & $e_{13}-e_{24}$ & $e_{14}+e_{23}$ & $e_{12}-e_{34}$ & $e_{13}+e_{24}$ & $e_{14}-e_{23}$\\ \hline
  $e_1$ & $-e_2$ & $-e_3$ & $-e_4$ & $-e_2$ & $-e_3$ & $-e_4$ \\ \hline
  $e_2$ & $e_1 $& $e_4 $& $-e_3 $& $e_1 $&$ -e_4 $&$ e_3$ \\ \hline
  $e_3$ & $-e_4 $& $e_1 $&$ e_2 $&$ e_4 $&$ e_1 $& $ -e_2$ \\ \hline
  $e_4 $& $e_3 $& $-e_2 $&$ e_1 $&$ -e_3 $&$ e_2 $&$ e_1 $\\
  \hline  
\end{tabular} 
\end{equation}

In a similar manner, any symmetric $(2,0)$ tensor b can be seen as an operator on the tangent space,
\begin{equation*}
b(X,Y)=\left\langle{b(X), Y}\right\rangle=\left\langle{X,b(Y)}\right\rangle=
\left\langle{b, X\wedge Y}\right\rangle.
\end{equation*}
Consequently, when b is viewed as a 1-form valued 1-form, $d_{\nabla}b$ denotes the exterior derivative (a 1-form valued 2-form). That is,
\begin{equation*}
(d_{\nabla}b)(X,Y,Z)=(\nabla b)(X,Y,Z)+(-1)^1 (\nabla b)(Y,X,Z)=\nabla_{X}b(Y,Z)-\nabla_{Y}b(X,Z).
\end{equation*}

Now we can define the fundamental tensors of our interest here, first via a local frame and then using the operator language. Let $\alpha \in \Lambda_2$,  $X,Y,Z\in TM$, and $\{e_{i}\}_{i=1}^{n}$ be a local normal orthonormal frame on a GRS $(M^n, g, f,\lambda)$. 
\begin{definition}
\label{framedef}
The tensors $P,Q,M,N: \Lambda_2 TM \otimes TM \rightarrow \mathbb{R}$ are defined as:
\begin{align}
\label{equaP}
P_{ijk}&=\nabla_{i}\Rc_{jk}-\nabla_{j}\Rc_{ik}=\nabla_{j}f_{ik}-\nabla_{i}f_{jk}=\RR_{jikp}\nabla^{p}f,
 \\
P(X\wedge Y,Z)&=-\RR(X,Y,Z,\nabla{f})=(d_{\nabla}\Rc)(X,Y,Z)=\delta{\RR}(Z,X,Y), \nonumber \\
P(\alpha, Z)&=\RR(\alpha, \nabla{f}\wedge Z)=\delta\RR(Z,\alpha);\nonumber\\
\label{equaQ}
Q_{ijk}&=g_{ki}\nabla_{j}\SS-g_{kj}\nabla_{i}\SS=2(g_{ki}\RR_{jp}-g_{kj}\RR_{ip})\nabla^{p}f,\\
Q(X\wedge Y,Z)&=2(X,Z)\Rc(Y,\nabla{f})-2(Y,Z)\Rc(X,\nabla{f}), \nonumber \\
Q(\alpha, Z)&=-2\Rc(\alpha(Z),\nabla{f})=-2\left\langle{\alpha{Z},\Rc(\nabla{f})}\right\rangle;\nonumber\\
\label{equaM}
M_{ijk}&=\RR_{kj}\nabla_{i}f-\RR_{ki}\nabla_{j}f,\\
M(X\wedge Y,Z)&=\Rc(Y,Z)\nabla_{X}{f}-\Rc(X,Z)\nabla_{Y}f=-\Rc((X\wedge Y)\nabla{f},Z), \nonumber\\
M(\alpha, Z)&=-\Rc(\alpha(\nabla{f}),Z)=-\left\langle{\alpha{\nabla{f}},\Rc(Z)}\right\rangle; \nonumber\\
\label{equaN}
N_{ijk}&=g_{kj}\nabla_{i}f-g_{ki}\nabla_{j}f,\\
N(X\wedge Y,Z)&=\left\langle{Y,Z}\right\rangle \nabla_{X}{f}-\left\langle{X,Z}\right\rangle \nabla_{Y}f=\left\langle{(X\wedge Y)Z, \nabla{f}}\right\rangle, \nonumber \\
N(\alpha,Z)&=\left\langle{\alpha{Z},\nabla{f}}\right\rangle=-\alpha(Z,\nabla{f}).\nonumber
\end{align}
\end{definition}

\begin{remark}
The tensors $P^{\pm},Q^{\pm},M^{\pm},N^{\pm}: \Lambda_2^{\pm}TM \otimes TM \rightarrow \mathbb{R}$ are defined by restricting $\alpha\in \Lambda_{2}^\pm TM$. They can be seen as operators on $\Lambda_{2}$ by standard projection.
\end{remark}

\begin{remark}
Before proceeding further, let us remark on the essence of these tensors. $P\equiv 0$ if and only if the curvature is harmonic; $Q\equiv 0$ if and only if the scalar curvature is constant; $N\equiv 0$ if and only if the potential function $f$ is constant; finally, $M\equiv 0$ if and only if either $\nabla{f}=0$ or $\Rc$ vanishes on the orthogonal complement of $\nabla{f}$.
\end{remark}

\subsection{Decomposition Lemmas}

Using the framework above, we now can represent the interior product  $i_{\nabla f}$ on components of the curvature tensor as follows. Again the Einstein summation convention is used here.
\begin{lemma} Let $(M,g,f,\lambda)$ be a GRS, for P, Q, M, N as in Definition \ref{framedef}, in a local normal orthonormal frame, we have
\label{interaction}
\begin{align}
\label{equaP1}
\RR_{ijkp}\nabla^{p}{f}&=\RR(e_i, e_{j}, e_{k}, \nabla f)=-P_{ijk}=\nabla^{p}\RR_{ijkp}=-\delta \RR(e_k, e_i, e_j),\\
(g\circ g)_{ijkp}\nabla^{p}f&=(g\circ g)(e_i, e_{j}, e_{k}, \nabla f)=-2N_{ijk},\\
(\Rc\circ g)_{ijkp}\nabla^{p}f&=(\Rc\circ g)(e_i, e_{j}, e_{k}, \nabla f)=\frac{1}{2}Q_{ijk}-M_{ijk},\\
\HH_{ijkp}\nabla^{p}f&=H(e_i, e_{j}, e_{k}, \nabla f)=M_{ijk}-\frac{1}{2}Q_{ijk}-2\lambda N_{ijk},\\
\label{equaWf}
\WW_{ijkp}\nabla^{p}f&=\WW(e_i, e_{j}, e_{k}, \nabla f)\\
&=-P_{ijk}-\frac{Q_{ijk}}{2(n-2)}+\frac{M_{ijk}}{(n-2)}-\frac{\SS N_{ijk}}{(n-1)(n-2)}. \nonumber 
\end{align}
\end{lemma}

\begin{proof}
The first formula is well-known (cf. \cite{ein}), following from the soliton equation and Bianchi identities. For the second, we compute,
\begin{align*}
(g\circ g)_{ijkp}\nabla^{p}f=& 2(g_{ik}g_{jp}-g_{ip}g_{jk})\nabla^{p}f\\
=&2g_{ik}\nabla_{j}f-2g_{jk}\nabla_i f=-2N_{ijk}.
\end{align*}
For the third, we use (\ref{rcandf}) to calculate
\begin{align*}
(\Rc\circ g)_{ijkp}\nabla^{p}f=&(\Rc_{ik}g_{jp}+\Rc_{jp}g_{ik}-\Rc_{ip}g_{jk}-\Rc_{jk}g_{ip})\nabla^p f\\
=&\Rc_{ik}\nabla_j f+\frac{1}{2}(g_{ik}\nabla_j \SS-g_{jk}\nabla_i \SS)-\Rc_{jk}\nabla_i f\\
=&\frac{1}{2}Q_{ijk}-M_{ijk}.
\end{align*}
The next formula is a consequence of the above formulas, definition of H (\ref{equaH}) and the soliton equation (\ref{GRS}). Finally, the last one comes from decomposition of the curvature operator (\ref{curvdecom}) and previous formulas; it appeared, for example, in \cite{cw11}.
\end{proof}

In addition, the divergence on these components can be written as linear combinations of $P, Q, M, N$. 

\begin{lemma} Let $(M,g,f,\lambda)$ be a GRS,  for P, Q, M, N  as in Definition \ref{framedef}, in a local normal orthonormal frame, we have 
\label{div}
\begin{align}
\nabla^{p}\RR_{ijkp}&=-P_{ijk},\\
\nabla^{p}(\SS g\circ g)_{ijkp}&=2Q_{ijk},\\
\label{equaHdiv}
\nabla^p(\Rc\circ g)_{ijkp}&=-\nabla^{p}\HH_{ijkp}=-P_{ijk}+\frac{1}{2}Q_{ijk},\\
\label{equaWdiv}
\nabla^{p}\WW_{ijkp}&=-\frac{n-3}{n-2}P_{ijk}-\frac{n-3}{2(n-1)(n-2)}Q_{ijk}:=-\frac{n-3}{n-2}C_{ijk}.
\end{align}
\end{lemma}

\begin{proof}
The first formula is well-known and comes from the second Bianchi identity \cite{ein}. For the second, we compute,
\begin{align*}
\nabla^{p}(\SS g\circ g)_{ijkp}&=2\nabla^{p}(\SS g_{ik}g_{jp}-\SS g_{ip}g_{jk})\\
&=2g_{ik} g_{jp}\nabla^p \SS-g_{jk}g_{ip}\nabla^{p}\SS\\
 &=2g_{ik}\nabla_j \SS - g_{jk}\nabla_i \SS=2Q_{ijk}.
\end{align*}
For the next one, we use (\ref{rcdiv}) to calculate,
\begin{align*}
\nabla^p(\Rc\circ g)_{ijkp}&= \nabla^p (\Rc_{ik}g_{jp}+\Rc_{jp}g_{ik}-\Rc_{ip}g_{jk}-\Rc_{jk}g_{ip})\\
&=g_{jp}\nabla^p\Rc_{ik}+g_{ik}\nabla^p\Rc_{jp}-g_{jk}\nabla^p\Rc_{ip}-g_{ip}\nabla^p\Rc_{jk}\\
&=\nabla_j\Rc_{ik}+\frac{1}{2}(g_{ik}\nabla_j\SS-g_{jk}\nabla_i\SS)-\nabla_i\Rc_{jk}\\
&=\frac{1}{2}Q_{ijk}-P_{ijk}.
\end{align*}
 Finally, the last one comes from decomposition of curvature (\ref{curvdecom}) and previous formulas; it also appeared in, for example, \cite[Eq. (9)]{Derd83}. 
\end{proof}
\begin{remark}
$C$ defined in (\ref{equaWdiv}) is also called the Cotton tensor in literature.
\end{remark}

\begin{remark}
\label{interactionpm} By the standard projection, and 
\begin{align*}
(\delta\WW)^{\pm}&=\delta(\WW^{\pm}), \\
(i_{\nabla{f}}\WW)^{\pm}&=i_{\nabla{f}}\WW^{\pm},
\end{align*}
the analogous identities hold if replacing $\WW, P, Q, M, N$ in Lemmas \ref{interaction} and \ref{div}  by $\WW^{\pm}, P^{\pm},Q^{\pm},M^{\pm},N^{\pm}$, respectively.
\end{remark}

The following observation is an immediate consequence of Lemma \ref{div}.

\begin{proposition} Let $(M^n,g,f,\lambda)$, $n>2$,  be a GRS and $\HH$ given by (\ref{equaH}). Then the tensor 
\[ F=\WW+\frac{n-3}{n-2}\HH+\frac{n(n-3)\SS}{4(n-1)(n-2)}g\circ g\]
is divergence free.
\end{proposition}
\begin{remark}
The result can be viewed as a generalization of the harmonicity of the Weyl tensor on an Einstein manifold.  
\end{remark}

Lastly, we introduce the following tensor $D$ which plays a crucial role in the classification problem (cf. \cite{caochen11}, \cite{caochen12}, \cite{cw11}),
\begin{align}
\label{equaD}
D_{ijk}&=-\frac{Q_{ijk}}{2(n-1)(n-2)}+\frac{M_{ijk}}{n-2}-\frac{\SS N_{ijk}}{(n-1)(n-2)}\\
&=C_{ijk}+\WW_{ijkp}\nabla^{p}f. \nonumber
\end{align}

\subsection{Norm Calculations}

\begin{lemma}
\label{product}
Let $(M,g,f,\lambda)$ be a GRS, then the following identities hold:
\begin{align*}
2\left\langle{P,Q}\right\rangle&=-|\nabla{\SS}|^2,\\
2\left\langle{P,N}\right\rangle&=\left\langle{\nabla{f}, \nabla{\SS}}\right\rangle,\\
2\left\langle{Q,Q}\right\rangle&=2(n-1)|\nabla{\SS}|^2,\\
2\left\langle{M,M}\right\rangle&=2|\Rc|^2|\nabla{f}|^2-\frac{1}{2}|\nabla{\SS}|^2,\\
2\left\langle{N,N}\right\rangle&=2(n-1)|\nabla{f}|^2,\\
2\left\langle{Q,M}\right\rangle&=|\nabla{\SS}|^2-2\SS \left\langle{\nabla{f},\nabla{\SS}}\right\rangle,\\
2\left\langle{Q,N}\right\rangle&=-2(n-1)\left\langle{\nabla{f},\nabla{\SS}}\right\rangle,\\
2\left\langle{M,N}\right\rangle&=2S|\nabla{f}|^2-\left\langle{\nabla{f},\nabla{\SS}}\right\rangle.
\end{align*}
Furthermore, if M is closed, then 
\begin{align*}
\int_{M}2\left\langle{P,P}\right\rangle e^{-f}=&\int_{M}|\nabla \Rc|^2 e^{-f},\\
\int_{M} 2\left\langle{P,M}\right\rangle=&2\int_{M}(\lambda|\Rc|^2-\Rc^3)+\int_{M}\left\langle{\nabla{f},\nabla |\Rc|^2}\right\rangle+\frac{1}{2}\int_{M}|\nabla\SS|^2.
\end{align*}
\end{lemma}
\begin{proof}
The main technique is to compute  under a normal orthonormal local frame. For example,
\begin{align*}
2\left\langle{P,Q}\right\rangle &=P_{ijk}Q_{ijk}\\
&=(\nabla_{i}\Rc_{jk}-\nabla_{j}\Rc_{ik})(g_{ki}\nabla_{j}\SS-g_{kj}\nabla_{i}\SS)\\
&=2(\nabla_{i}\Rc_{jk}-\nabla_{j}\Rc_{ik})g_{ki}\nabla_{j}\SS\\
&=2\nabla_{j}\SS(\nabla_{k}\Rc_{kj}-\nabla_{j}\Rc_{kk})\\
&=|\nabla \SS|^2-2|\nabla \SS|^2=-|\nabla\SS|^2.
\end{align*}
Other equations follow from similar calculation. \\

When M is closed, we can integrate by parts.  In particular, the first equation was first derived in \cite{ein}. For the second, we compute that
\begin{align*}
\int_{M}2\left\langle{P,M}\right\rangle&=2\int_{M}(\nabla_{i}\Rc_{jk}-\nabla_{j}\Rc_{ik})\Rc_{kj}\nabla_{i}f\\
&=\int_{M}\nabla_{i}f\nabla_{i}\Rc_{jk}^2-2\int_M \nabla_{j}\Rc_{ik}\Rc_{kj}\nabla_{i}f,\\
\int_M \nabla_{j}\Rc_{ik}\Rc_{kj}\nabla_{i}f&=-\int_{M}\Rc_{ik}\Rc_{kj}f_{ij}-\int_{M}\Rc_{ik}f_{i}\nabla_{j}\Rc_{kj}\\
&=-\int_{M}(\lambda|\Rc|^2-\Rc^3)-\frac{1}{4}\int_{M} |\nabla\SS|^2. 
\end{align*}
Hence, the statement follows.
\end{proof}

\begin{remark} The factor of 2 is due to our convention of calculating norm. Some special cases of dimension four also appeared in \cite[Proposition 4]{bs11}.
\end{remark}

An interesting consequence of the above calculation  is the following corollary, which exposes the orthogonality of $Q, N$ versus $i_{\nabla f}\WW$, $\delta\WW$. 

\begin{corollary}
\label{QD0}
  Let $(M,g,f,\lambda)$ be a GRS. 
  
{\bf a.} At each point, we have
\begin{equation*}
0=\left\langle{Q, i_{\nabla{f}}\WW}\right\rangle=\left\langle{N,i_{\nabla f}\WW}\right\rangle=\left\langle{Q,\delta\WW}\right\rangle=\left\langle{N,\delta\WW}\right\rangle.
\end{equation*}

{\bf b.} If $M$ is closed, then,
\begin{equation}
\int_{M} 2|\delta \WW|^2 e^{-f} =(\frac{n-3}{n-2})^2 \int_{M}(|\nabla \Rc|^2-\frac{1}{(n-1)}|\nabla \SS|^2) e^{-f}.
\end{equation}
\end{corollary}
\begin{proof}
Part a) follows immediately from Lemmas \ref{interaction}, \ref{div}, \ref{product}, and our convention (\ref{productconv}). For example,
\begin{align*}
\left\langle{Q, i_{\nabla{f}}\WW}\right\rangle=&\sum_{i<j} Q_{ijk}(i_{\nabla f}\WW)_{kij}\\
=&\sum_{i<j} Q_{ijk}\nabla^p f \WW_{pkij}=-\sum_{i<j} Q_{ijk} \WW_{ijkp}\nabla^p f\\
=&\left\langle{Q, P+\frac{Q}{2(n-2)}-\frac{M}{n-2}+\frac{\SS N}{(n-1)(n-2)}}\right\rangle\\
=&-\frac{|\nabla \SS|^2}{2}+\frac{(n-1)|\nabla \SS|^2}{2(n-2)}-\frac{|\nabla \SS|^2}{2(n-2)}+
\frac{\SS \left\langle{\nabla f, \nabla \SS}\right\rangle}{n-2}-\frac{(n-1)\SS \left\langle{\nabla f, \nabla \SS}\right\rangle}{(n-1)(n-2)}\\
=&0.
\end{align*}

Other formulas follow from similar calculations.

For part b) we observe that,
\begin{align*}
|\delta \WW|^2=&(\frac{n-3}{n-2})^2 \left\langle{P+\frac{Q}{2(n-1)},P+\frac{Q}{2(n-1)}}\right\rangle\\
=&(\frac{n-3}{n-2})^2 \left\langle{P+\frac{Q}{2(n-1)},P}\right\rangle.
\end{align*}
Notice that we apply part a) in the last step. Consequently, applying Lemma \ref{product} again yields
\begin{align*}
2\int_{M} |\delta \WW|^2 e^{-f} &= (\frac{n-3}{n-2})^2 \int_{M} 2\left\langle{P+\frac{Q}{2(n-1)},P}\right\rangle e^{-f}\\
&=(\frac{n-3}{n-2})^2 \int_{M}(|\nabla\Rc|^2-\frac{|\nabla \SS|^2}{2(n-1)})e^{-f}.
\end{align*}
\end{proof}

\begin{remark}
Part b) recovers the well-known fact that harmonic curvature implies harmonic Weyl tensor and constant scalar curvature.
\end{remark}

Now we are ready to prove Theorem \ref{intanddiv}.
\begin{proof}{\bf (Theorem \ref{intanddiv})}
First, we observe,
\begin{align*}
\left\langle{\WW, \He{f}\circ \He{f}}\right\rangle&=\sum_{i<j, k<l}\WW_{ijkl}(\He{f}\circ \He{f})_{ijkl}\\
&=\frac{1}{2}\sum_{k<l; i,j}\WW_{ijkl}(\He{f}\circ \He{f})_{ijkl}\\
&=\sum_{k<l; i,j}\WW_{ijkl}({f}_{ik}{f}_{jl}-{f}_{il}{f}_{jk})\\
&=\sum_{i,j,k,l}\WW_{ijkl}{f}_{ik}{f}_{jl}.
\end{align*}

Next, subduing the summation notation, we integrate by parts,
\[\int_{M}\WW_{ijkl}f_{ik}f_{jl}=-\int_{M} \nabla_{i}\WW_{ijkl}f_{k}f_{jl}-\int_{M}\WW_{ijkl}f_{k}\nabla_{i}f_{jl}.\]

The first term can be written as
\begin{align*}
\int_{M} \nabla_{i}\WW_{ijkl}f_{k}f_{jl}&=\int_{M} \nabla_{i}\WW_{ijkl}f_{k}(\lambda g_{jl}-\Rc_{jl})\\
&=-\int_{M}\nabla_{i}\WW_{ijkl}f_{k}\Rc_{jl}=-\frac{1}{2}\int_{M}(\delta\WW)_{jkl}M_{klj}\\
&=-\int_{M}\left\langle{\delta\WW,M}\right\rangle.
\end{align*}

Next, we compute the second term,
\begin{align*}
\int_{M}\WW_{ijkl}f_{k}\nabla_{i}f_{jl}&=-\int_{M}\WW_{ijlk}f_{k}\nabla_{i}(g_{jl}-\Rc_{jl})= \int_{M}\WW_{ijlk}f_{k}\nabla_{i}\Rc_{jl}\\
&=\frac{1}{2}\int_{M}\WW_{ijlk}f_{k}P_{ijl}=-\int_{M}\left\langle{i_{\nabla f}\WW,P+\frac{Q}{2(n-1)}}\right\rangle\\
&=-\frac{n-2}{n-3}\int_{M}\left\langle{\delta\WW, i_{\nabla f}\WW}\right\rangle\\
&=\frac{n-2}{n-3}\int_{M}\left\langle{\delta\WW, -P+\frac{M}{n-2}}\right\rangle.
\end{align*}
It is noted that we have used Corollary \ref{QD0} repeatedly to manipulate $Q$ and $N$. To conclude, we combine equations above,
\begin{align*}
\int_{M}\WW_{ijkl}f_{ik}f_{jl}&=\int_{M}\left\langle{\delta\WW,M}\right\rangle-\frac{n-2}{n-3}\int_{M}\left\langle{\delta\WW, -P+\frac{M}{n-2}}\right\rangle\\
&=\frac{1}{n-3}\int_{M}\left\langle{\delta\WW, (n-2)P+(n-4){M}}\right\rangle.
\end{align*}
If $n=4$,  then 
\begin{align*}
\int_{M}\WW_{ijkl}f_{ik}f_{jl}&=\int_{M}2\left\langle{\delta\WW, P}\right\rangle=\int_{M} 2\left\langle{\delta\WW, P+\frac{Q}{6}}\right\rangle\\
&=\int_{M}2\left\langle{\delta\WW, 2\delta\WW}\right\rangle=4\int_{M}|\delta\WW|^2.
\end{align*}
\end{proof}
\begin{remark} The formula in dimension four is also a consequence of the divergence-free property of the Bach tensor. We omit the details here.
\end{remark}

Moreover, in dimension four, we have similar results for $\WW^{\pm}$. 
\begin{lemma}
\label{4dimQD0}
 Let $(M^4,g,f,\lambda)$ be a GRS, then at each point, we have
\begin{equation}
0=\left\langle{Q^{\pm},i_{\nabla f}\WW^{\pm}}\right\rangle=\left\langle{Q^{\pm},\delta\WW^{\pm}}\right\rangle=\left\langle{N^{\pm}, i_{\nabla f}\WW^{\pm}}\right\rangle=\left\langle{N^{\pm},\delta\WW^{\pm}}\right\rangle.
\end{equation}
\end{lemma}
\begin{proof}
It suffices to show the statements is true for the self-dual part. 

Let $\{e_{i}\}_{i=1}^{4}$ be a normal orthonormal local frame and let $\{\alpha_{i}\}_{i=1}^4$ be  an orthonormal basis for $\Lambda_2^{+}$ . Then 
\begin{align*}
\left\langle{Q^{+},i_{\nabla f}\WW^{+}}\right\rangle=&\sum_{i}\sum_{j}Q(\alpha_i,e_j)\WW(\nabla{f}\wedge e_j,\alpha_i)\\
=&-2\left\langle{\alpha_i(e_j),\Rc(\nabla f)}\right\rangle \WW(\nabla{f}\wedge e_j,\alpha_i).
\end{align*}
Furthermore, we can choose a special basis, namely the normal form as in (\ref{normalform}). Then $\alpha_{i}$'s diagonalize $\WW^{+}$ with eigenvalues $\lambda_i$'s. Consequently,
\begin{equation*}
\WW(\nabla{f}\wedge e_j,\alpha_i)=\lambda_i\alpha_i(\nabla{f}\wedge e_j)=\lambda_i\left\langle{\nabla{f},\alpha_i(e_j)}\right\rangle.
\end{equation*}  
Thus,
\begin{align*} \left\langle{Q^{+},i_{\nabla f}\WW^{+}}\right\rangle=&-2\lambda_i\left\langle{\alpha_i(e_j),\Rc(\nabla f)}\right\rangle \left\langle{\alpha_i(e_j),\nabla{f}}\right\rangle\\
=&-2\eta_k\left\langle{e_k,\Rc(\nabla f)}\right\rangle \left\langle{e_k,\nabla{f}}\right\rangle,\\
\text{for } \eta_k=&\sum_{i,j: \alpha_i(e_j)=\pm e_k}\lambda_{i}.
\end{align*}
Now by (\ref{bivectoract}), it is easy to see that each $\eta_k=0$ because $\WW^{+}$ is traceless. \\

Next, we state the following fact. 

{\bf Claim:} $\left\langle{P^{+},Q^{+}}\right\rangle=-\frac{1}{4}|\nabla \SS|^2.$

To prove this claim, we choose $\{\alpha_i\}$ as in (\ref{Hodgebase}) and observe that, 
\begin{align*}
P(\alpha_1, e_j)Q(\alpha_1, e_j)=& \frac{1}{2}P(e_{12}+e_{34},e_j)Q(e_{12}+e_{34},e_j)\\
=&-(P_{12j}+P_{34j})\left\langle{(e_{12}+e_{34})e_j,\Rc(\nabla{f})}\right\rangle\\
=&-(\nabla_1\Rc_{2j}-\nabla_{2}\Rc_{1j}+\nabla_{3}\Rc_{4j}-\nabla_4\Rc_{3j})\left\langle{(e_{12}+e_{34})e_j,\Rc(\nabla{f})}\right\rangle.\\
\end{align*}

Similarly,
\begin{align*}
P(\alpha_2, e_j)Q(\alpha_2, e_j)
=&-(\nabla_1\Rc_{3j}-\nabla_{3}\Rc_{1j}-\nabla_{2}\Rc_{4j}+\nabla_4\Rc_{2j})\left\langle{(e_{13}-e_{24})e_j,\Rc(\nabla{f})}\right\rangle,\\
P(\alpha_3, e_j)Q(\alpha_3, e_j)
=&-(\nabla_1\Rc_{4j}-\nabla_{4}\Rc_{1j}+\nabla_{2}\Rc_{3j}-\nabla_3\Rc_{2j}) \left\langle{(e_{14}+e_{23})e_j,\Rc(\nabla{f})}\right\rangle.
\end{align*}

Thus, 
\begin{align*}
\left\langle{P^{+},Q^{+}}\right\rangle=&\sum_{i,j}P(\alpha_i,e_j)Q(\alpha_i,e_j)\\
=&-\sum_{k}\zeta_{k}\left\langle{e_k,\Rc(\nabla f)}\right\rangle,\\
\text{for } \zeta_k=&\sum_{i,j: \alpha_{i}(e_j)=e_k}\sqrt{2}P(\alpha_i, e_j)-\sum_{i,j: \alpha_{i}(e_j)=-e_k}\sqrt{2}P(\alpha_i, e_j).
\end{align*}
Using (\ref{bivectoract}), we can compute, 
\begin{align*}
\zeta_{1}=&\sqrt{2}\Big(P(\alpha_1, e_2)+P(\alpha_2, e_3)+P(\alpha_3, e_4)\Big)\\
=&\nabla_1\Rc_{22}-\nabla_{2}\Rc_{12}+\nabla_{3}\Rc_{42}-\nabla_4\Rc_{32}\\
&+\nabla_1\Rc_{33}-\nabla_{3}\Rc_{13}-\nabla_{2}\Rc_{43}+\nabla_4\Rc_{23}\\
&+\nabla_1\Rc_{44}-\nabla_{4}\Rc_{14}+\nabla_{2}\Rc_{34}-\nabla_3\Rc_{24}\\
=&\nabla_1(\SS-\Rc_{11})-(\frac{1}{2}\nabla_1\SS-\nabla_1\Rc_{11})=\frac{1}{2}\nabla_{1}\SS.
\end{align*} 
Similarly we have $\zeta_{k}=\frac{1}{2}\nabla_k\SS$. We also have $\Rc(\nabla f)=\frac{1}{2}\nabla \SS$. This proves our claim.\\

In addition, it is easy to see that 
\[\left\langle{Q^{+},Q^{+}}\right\rangle=\frac{3}{2}|\nabla \SS|^2.\]

Since $\delta\WW^{+}=\frac{P^{+}}{2}+\frac{Q^+}{12}$, it follows that 
\[ \left\langle{Q^{+},\delta\WW^{+}}\right\rangle=0.\]

The statements involved N follow from analogous calculations as $$N(\alpha_i, e_j)=\left\langle{\alpha_i(e_j),\nabla{f}}\right\rangle.$$
\end{proof}

By manipulation as in the proof of Theorem \ref{intanddiv}, using Remark \ref{interactionpm} (replacing Lemmas \ref{interaction} and \ref{div}) and Lemma \ref{4dimQD0} (replacing Lemma \ref{QD0}), we immediately obtain the following result.   
\begin{corollary}
\label{intanddivpluscor}
Let $(M,g,f,\lambda)$ be a four-dimensional closed GRS. Then we have the following identity:
\begin{equation}
\int_{M}\left\langle{\WW^+, \Rc\circ \Rc}\right\rangle=4\int_{M}|\delta \WW^{+}|^2.
\end{equation}
\end{corollary}



\section{\textbf{Rigidity Results}}

In this section, we present conditions that imply the rigidity  of a GRS using the analysis on the framework discussed in the previous section.\\

First, Proposition \ref{Dequal0} provides a geometrical way to understand tensor D defined in (\ref{equaD}). In particular, it says that $D\equiv 0$ is equivalent to a special condition, namely, the normalization of $\nabla{f}$ (if not trivial) is an eigenvector of the Ricci tensor, and all other eigenvectors have the same eigenvalue. Such a structure will imply rigidity as the geometry of the level surface (of $f$) being well-described. 

On the other hand, Theorem \ref{intanddiv} reveals an interesting connection between the Ricci tensor and the Weyl tensor in dimension four. That allows us to obtain rigidity results using only the structure of the Ricci curvature for a GRS.

\begin{theorem}
\label{rigidRc}
Let $(M^4,g,f,\lambda)$ be a closed four-dimensional GRS. Assume that at each point the Ricci curvature has one eigenvalue of multiplicity one and another of multiplicity three, then the GRS is rigid, hence Einstein. 
\end{theorem}

We also find conditions that imply the vanishing of tensor $D$.

\begin{theorem}
\label{implyD}
Let $(M^n,g,f,\tau)$, $n>3$, be a GRS. Assuming one of these conditions holds:
\begin{enumerate}
\item $i_{\nabla{f}} \Rc \circ g\equiv 0$;
\item $i_{\nabla{f}} \WW \equiv 0$ and $\delta\WW(.,.,\nabla{f})=0$.
\end{enumerate}
 Then at the point $\nabla{f}\neq 0$, $D=0$. 
\end{theorem}

\begin{remark} $D\equiv 0$ can be derived from other conditions such as the vanishing of the Bach tensor (cf. \cite[Lemma 4.1]{caochen11}). 
\end{remark}

\begin{remark} For GRS's, condition (2) is a slight improvement of \cite{ca12}, where the author characterizes generalized quasi-Einstein manifolds with $\delta \WW=i_{\nabla{f}}\WW=0$. 
\end{remark}


In dimension four, the result can be improved significantly. 

\begin{theorem}
\label{improved4}
  Let $(M,g,f,\lambda)$ be a four-dimensional GRS. At points where $\nabla{f}\neq 0$, then $\WW^{+}(\nabla{f},.,.,.)=0$ implies $\WW^{+}=0$.
\end{theorem}

As discussed in the last section, there are some similarities between taking the divergence and interior product $i_{\nabla{f}}$ of the Weyl tensor, for example, see Corollary \ref{QD0}. The following theorem is inspired by condition (1) of Theorem \ref{implyD}.

\begin{theorem}
\label{deltaRc}
Let $(M^n,g,f,\tau)$, $n>3$, be a GRS.  Then $\delta (\Rc\circ g)\equiv 0$ if and only if  the Weyl tensor is harmonic and the scalar curvature is constant. 
\end{theorem}


An immediate consequence of the results above (plus known classifications discussed in the Introduction) is to obtain rigidity results.

\begin{corollary}
\label{cor5dim}
 Let $(M^{n},g,f,\lambda)$, $n\geq 4$, be a complete shrinking GRS.
 
 {\bf i.} If $i_{\nabla{f}}\Rc\circ g\equiv 0$, then $(M^{n},g,f,\lambda)$ is Einstein;
 
 {\bf ii.} If $i_{\nabla{f}}\WW=0$ and $\delta\WW(.,.,\nabla{f})=0$, then $(M^{n},g,f,\lambda)$ is rigid of rank $k=0, 1, n$;
 
  {\bf iii.}If $\delta (\Rc\circ g)=0$, then $(M^{n},g,f,\lambda)$ is rigid of rank $0\leq k\leq n$.
 
\end{corollary}
In particular, when the dimension is four, we have the following result.
 
\begin{corollary}
\label{cor4dim}
 Let $(M,g,f,\lambda)$ be a four-dimensional complete GRS. If $$\WW^{+}(\nabla{f},.,.,.)=0,$$ then the GRS is either Einstein or has $\WW^{+}=0$. Furthermore, in the second case, it is isometric to a Bryant soliton or Ricci flat manifold if $\lambda=0$; or is a finite quotient of $\RR^4$, $\SS^3\times \RR$, $\SS^4$ or $CP^2$ if $\lambda>0$.
\end{corollary}

The general strategy to prove aforementioned statements is to use the framework to study the structure of the Ricci tensor.

\subsection{Eigenvectors of the Ricci curvature}

Here we study various interconnections between the eigenvectors of the Ricci curvature, the Weyl tensor, and the potential function. First, we observe the following lemma.

\begin{lemma}
\label{rigidRc1} 
Let $(M,g)$ be a Riemannian manifold. Assume that, at each point, the Ricci curvature has one eigenvalue of multiplicity one and another of multiplicity $n-1$. Then we have,
\begin{equation*}
\left\langle{\WW, \Rc\circ \Rc}\right\rangle=0.
\end{equation*}
\end{lemma}
\begin{proof}
Without loss of generality, we can choose a basis $\{e_{i}\}_{i=1}^n$ of $T_{p}M$ consisting of eigenvectors of $\Rc$, namely $\Rc_{11}=\eta$ and $\Rc_{ii}=\zeta$ for $i=2,...,n$. Then, 
\begin{align}
\left\langle{\WW, \Rc\circ \Rc}\right\rangle&=\sum_{i<j; k<l}\WW_{ijkl}\Rc_{ik}\Rc_{jl}\\
&=\sum_{i<j}\WW_{ijij}\Rc_{ii}\Rc_{jj}=\eta\zeta \sum_{j}\WW_{1j1j}+\zeta^2\sum_{1<i<j}\WW_{ijij}.
\end{align}
We observe that,
\begin{align}
\sum_{j>1}\WW_{ijij}&=-\WW_{1i1i},\\
2\sum_{1<i<j}\WW_{ijij}&=\sum_{i>1}\sum_{j>1}\WW_{ijij}=-\sum_{i}\WW_{1i1i}=0.
\end{align}
The result then follows.
\end{proof}

Next, a consequence of our previous framework (on P, Q, M, and N) is the following characterization about the condition  $\Rc(\nabla f)=\mu \nabla f$.

\begin{lemma}
\label{gradeigen}
Let $(M,g,f,\lambda)$ be a GRS. Then the followings are equivalent:
\begin{enumerate}
\item $\Rc(\nabla{f})=\mu \nabla{f}$;
\item $Q(.,.,\nabla{f})=0$;
\item $M(.,.,\nabla{f})=0$;
\item $\delta\WW(\nabla{f},.,.)=0$; 
\item $\delta \HH(\nabla{f},.,.)=0$.
\end{enumerate}
\end{lemma}
\begin{proof}

We\rq{}ll show that $(1)\leftrightarrow (2)$, $(1)\leftrightarrow (3)$, $(2)\leftrightarrow (4)$, and $(2)\leftrightarrow (5)$.

For $(2)\rightarrow (1)$: Let $\alpha\in \Lambda_{2}$, we have $0=Q(\alpha,\nabla{f})=-2(\alpha(\nabla{f}), \Rc(\nabla{f}))$. Since $\alpha$ can be arbitrary, $\alpha(\nabla{f})$ can realize any vector in the complement of $\nabla{f}$ in $TM$. Therefore, $\Rc(\nabla{f})=\mu \nabla{f}$.

For $(1)\rightarrow (2)$: $Q(\alpha,\nabla{f})=-2(\alpha(\nabla{f}), \Rc(\nabla{f}))=-2(\alpha(\nabla{f}), \mu\nabla{f})=0$ because $\alpha(\nabla{f})\perp \nabla{f}$.

$(1)$ being equivalent to $(3)$ follows from an identical argument as above.

$(2)$ being equivalent to $(4)$ follows from
\begin{align*}
\delta\WW(X,Y,Z)=&\frac{n-3}{n-2}P(Y,Z,X)+\frac{n-3}{2(n-1)(n-2)}Q(Y,Z,X),\\
P(Y,Z,\nabla{f})=&-\RR(Y,Z,\nabla{f},\nabla{f})=0.
\end{align*}

$(2)$ being equivalent to $(5)$ follows from
\begin{align*}
\delta\HH(X,Y,Z)=&-P(Y,Z,X)+\frac{1}{2}Q(Y,Z,X),\\
P(Y,Z,\nabla{f})=&-\RR(Y,Z,\nabla{f},\nabla{f})=0.
\end{align*} 
\end{proof}

Furthermore, the rigidity of these operators $Q, M, N$ is captured by the following result.

\begin{proposition}
\label{genvan}
Let $(M^n,g,f,\tau)$, $n>3$, be a GRS and $T=a Q+ b M+ cN$ for some real numbers a,b,c.
 
{\bf i.} Assume that $T\equiv 0$. If $a\neq 0$ then $\Rc(\nabla{f})=\mu \nabla f$; moreover, if $\nabla{f}\neq 0$ and $b \neq 0$, then all other eigenvectors must have the same eigenvalue;

{\bf ii.} In dimension four, if $T_{\mid \Lambda_{2}^{+} \otimes TM}\equiv 0$ then $T\equiv 0$.  
\end{proposition}

\begin{proof}
Let $\{e_{i}\}_{i=1}^n$ be an orthonormal basis which consists of eigenvector of $\Rc$ with corresponding eigenvalues $\lambda_{i}$. Then we have
\begin{align}
T(\alpha, e_{i})&= aQ(\alpha,e_{i})+bM(\alpha,e_{i})+cN(\alpha,e_{i}) \nonumber \\
&=-2a\left\langle{\alpha(e_{i}),\Rc(\nabla{f})}\right\rangle-b\left\langle{\alpha(\nabla{f}),\Rc(e_{i})}\right\rangle+ c\left\langle{\alpha(e_{i}),\nabla{f}}\right\rangle \nonumber \\
&=-2a\left\langle{\alpha(e_{i}),\Rc(\nabla{f})}\right\rangle+b\left\langle{\nabla{f},\alpha(\lambda_{i}e_{i})}\right\rangle+c\left\langle{\alpha(e_{i}),\nabla{f}}\right\rangle \nonumber \\
\label{genvan11}
&=\left\langle{\alpha(e_{i}),-2a\Rc(\nabla{f})+b\lambda_{i}\nabla{f}+c\nabla{f}}\right\rangle.
\end{align}

{\bf i.} Without loss of generality, we can assume $\nabla f\neq 0$. Since $T(\alpha, e_{i})= 0$ for arbitrary $\alpha$ and $e_{i}$,  
\[T(\alpha, \nabla{f})=0=\left\langle{\alpha(\nabla{f}),\Rc(\nabla{f})}\right\rangle=Q(\alpha,\nabla{f}).\]
By Lemma \ref{gradeigen},  $e_{1}=\frac{\nabla{f}}{|\nabla{f}|}$ is an eigenvector of $\Rc$. Plugging into (\ref{genvan11}) yields,
\[T(\alpha, e_{i})=(-2a\lambda_{1}+b\lambda_{i}+c)\left\langle{\alpha(e_{i}),\nabla{f}}\right\rangle.\]
Therefore, $-2a\lambda_{1}+b\lambda_{i}+c=0$. Hence, as $b\neq 0$, all other eigenvectors have the same eigenvalue.

{\bf ii.}    In dimension four, fix a unit vector $e_{i}$ and note that $T(\alpha,e_{i})=0$ for any $\alpha \in \Lambda_{2}^{+}$. By Lemma \ref{decom} and Remark \ref{derem}, $T(\beta,e_{i})=0$ for all $\beta \in \Lambda_{2}^{-}$. As $e_{i}$ is arbitrary the result then follows. 

\end{proof}


Recall that tensor $D$ is a special linear combination of $M, N, Q$. Therefore, we obtain the following geometric characterization.

\begin{proposition}
\label{Dequal0}
Let $(M^n,g)$, $n>3$, be a Riemannian manifold and D defined as in (\ref{equaD}). Then the followings are equivalent:
\begin{enumerate}
\item $D\equiv 0$;
\item The Weyl tensor under the conformal change $\tilde{g}=e^{\frac{-2f}{n-2}}g$ is harmonic;
\item Either $\nabla{f}=0$ and Cotton tensor $C_{ijk}=0$, or $\nabla{f}$ is an eigenvector of $\Rc$ and all other eigenvectors have the same eigenvalue.  
\end{enumerate}
 \end{proposition}  
 
 \begin{proof} 
 We shall show  {\bf $(1)\leftrightarrow (2)$}, {\bf $(1)\rightarrow (3)$} and  {\bf $(3)\rightarrow (1)$}.
 
 For {\bf $(1)\leftrightarrow (2):$} By equation (\ref{equaD}) and (\ref{equaWdiv}), we have
\[ D_{ijk}=C_{ijk}+\WW_{ijkp}\nabla^{p}f=\frac{n-2}{n-3}(\delta \WW)_{kij}-\WW(\nabla{f}, e_{k}, e_{i}, e_{j}).\]

Thus, $D\equiv 0$ is equivalent to 
\begin{equation*}
\label{Dequal01}
\delta\WW(X,Y,Z)-\frac{n-3}{n-2}\WW(\nabla{f},X,Y,Z)=0.
\end{equation*}

Under the conformal transofrmation $\tilde{g}=u^2 g$ (see the appendix), $\widetilde{\WW}=u^2 \WW$, and
\[ \delta \widetilde{\WW}(X,Y,Z)=\delta\WW(X,Y,Z)+(n-3)\WW(\frac{\nabla{u}}{u},X,Y,Z).\]

The result then follows from the last two equation. 

The statement {\bf $(1)\rightarrow (3)$} follows from \cite[Proposition 3.2 and Lemma 4.2]{caochen11}.

For {\bf $(3)\rightarrow (1)$:} $\forall a, b, c,$ let $T=aQ+bM+cN$. For any $\alpha\in \Lambda_{2}$ and $e_{i}$ a unit tangent vector, by (\ref{genvan11}), we have
 \begin{align*}
 T(\alpha, e_{i})&=\left\langle{\alpha(e_{i}),-2a\Rc(\nabla{f})+b\lambda_{i}\nabla{f}+c\nabla{f}}\right\rangle.
 \end{align*}
 For the tensor D, 
 \begin{align*}
 a=&\frac{-1}{2(n-1)(n-2)},\\
 b=&\frac{1}{n-2},\\
 c=&\frac{-\SS}{(n-1)(n-2)}.
 \end{align*}
 
  If $\nabla f=0$ then $T\equiv 0$, hence $D\equiv 0$. If $\nabla{f}\neq 0$, then there exist $e_{1}=\frac{\nabla{f}}{|\nabla{f}|}$ and $\{e_{i}\}_{i=2}^{n}$, eigenvectors of $\Rc$, with eigenvalues $\zeta, \eta$, respectively. Then, 
 \[T(\alpha, e_{i})=\left\langle{\alpha(e_{i}), (-2a\zeta+b\eta+c)\nabla{f}}\right\rangle.\]

 Since $\zeta+(n-1)\eta=\SS$, with given values of a, b, c above, it follows that $-2a\zeta+b\eta+c=0$. Thus, $D\equiv 0$.
  \end{proof}
 
 \begin{remark}
 Note that our formulas are different from \cite[2.19]{Derd83} by a sign convention. 
 \end{remark}
\begin{remark} Under that conformal change of the metric, the Ricci tensor is given by
 \begin{align*}
 \widetilde{\Rc}=&\Rc+\text{Hess}f+\frac{1}{n-2}df\otimes df+\frac{1}{n-2}(\Delta f-|\nabla{f}|^2)g\\
 =&\frac{1}{n-2}df\otimes df +\frac{1}{n-2}(\Delta f-|\nabla{f}|^2+(n-2)\lambda)g.
 \end{align*}
 Therefore, at each point, $\widetilde{\Rc}$ has at most two eigenvalues. 
 Furthermore, since $\tilde{g}$ has harmonic Weyl tensor, its Schouten tensor 
 \begin{equation*}
 \widetilde{\text{Sc}}=\frac{1}{n-2}(\widetilde{\Rc}-\frac{1}{2(n-1)}\tilde{\SS}\tilde{g})
 \end{equation*}
 is a Codazzi tensor with at most two eigenvalues. Using the splitting results for Riemannian manifolds admitting such a tensor gives another proof of results in \cite{caochen11}. This method is inspired by \cite{ca12}.
\end{remark}

Now we investigate several conditions which will imply that $\Rc(\nabla{f})=\mu \nabla{f}$. 

%

\begin{proposition}
\label{gradeigen1}
Let $(M^n,g,f,\tau)$, $n>3$, be a GRS. Assuming one of these conditions holds:
\begin{enumerate}
\item $i_{\nabla{f}}\WW \equiv 0$;
\item  $\delta \WW^{+}=0$ if $n=4$.
\end{enumerate}
 Then $\Rc(\nabla{f})=\mu \nabla{f}$.  
\end{proposition}

\begin{proof} The idea is to find a connection of each condition with Lemma \ref{gradeigen}.

{\bf Assuming (1):} We claim that $\delta\WW(\nabla{f},.,.)=0$. 

Choosing a normal local frame $\{e_{i}\}_{i=1}^n$, we have:
\begin{align*}
\delta\WW(\nabla{f},e_{k},e_{l})=& \sum_{i}(\nabla_{i}\WW)(e_{i}, \nabla{f}, e_{k}, e_{l})\\
=&\sum_{i}\nabla_{i} \WW(e_{i}, \nabla{f}, e_{k}, e_{l})-\sum_{i}\WW(e_{i}, \nabla_{i}\nabla{f}, e_{k}, e_{l})\\
=&0-\WW(\text{Hess}f,e_{k},e_{l}).
\end{align*}

Since $\He{f}$ is symmetric and $\WW$ is anti-symmetric, 
$\delta\WW(\nabla{f},.,.)=0$. The result then follows.\\

{\bf Assuming (2):} First recall
\[ \delta\WW(X,Y,Z)=\frac{1}{2}C(Y,Z,X)=\frac{1}{2}P(Y\wedge Z,X)+\frac{1}{12}Q(Y\wedge Z, X). \]

$\forall \alpha \in \Lambda^{2}_{+}$, since 
\[ \delta \WW^{-}(X,\alpha)=\nabla_{i}\WW^{-}(e_{i}\wedge X, \alpha)=0, \]

we have 
\begin{equation*}
\label{eigen1}
 \delta(\WW)(X,\alpha)= \delta(\WW^{+})(X,\alpha)=\frac{1}{2}P(\alpha,X)+\frac{1}{12}Q(\alpha, X). 
\end{equation*}

Since $0=\RR(Y,Z,\nabla{f},\nabla{f})=-P(Y\wedge Z, \nabla{f})$ and $\delta \WW^{+}=0$, hence $Q(\alpha, \nabla{f})=0$. The desired statement follows from Lemmas \ref{decom} and \ref{gradeigen}.

\end{proof}

\subsection{Proofs of Rigidity Theorems}

\begin{proof}{\bf (Theorem \ref{rigidRc})}

By Lemma \ref{rigidRc1}, we have $$\int_{M}\WW(\Rc\circ \Rc)=0.$$ Theorem \ref{intanddiv}, hence, implies that $\delta \WW\equiv 0$. Then by the rigidity result for harmonic Weyl tensor discussed in the Introduction, the result follows.
\end{proof}


\begin{proof}{\bf (Theorem \ref{implyD})}.

{\bf Assuming (1):} We observe that 
\[\Rc\circ g(X,Y,Z,\nabla{f})=\frac{1}{2}Q(X,Y,Z)-M(X,Y,Z).\]

Therefore, the result follows from Lemma \ref{genvan} and Proposition \ref{Dequal0}.\\

{\bf Assuming (2):} By Proposition \ref{gradeigen1}, $e_{1}=\frac{\nabla{f}}{|\nabla{f}|}$ is a unit eigenvector.  Let $\{e_{i}\}_{i=1}^n$ be an orthonomal basis of $\Rc$ with eigenvalues $\lambda_{i}$.
By (\ref{equaWf}) and $\WW(\nabla{f},.,.,.)=0$, 
\[P=-\frac{Q}{2(n-2)}+\frac{M}{(n-2)}-\frac{\SS N}{(n-1)(n-2)}.\]

Thefore,  
\begin{align}
\label{ran1}
P(i,j,k)&=\frac{|\nabla{f}|}{n-2}\Big[\lambda_{1}(\delta_{jk}\delta_{1i}-\delta_{ik}\delta_{j1})-\lambda_{k}(\delta_{j1}\delta_{ik}-\delta_{i1}\delta_{jk})-\frac{S}{n-1}(\delta_{jk}\delta_{1i}-\delta_{ik}\delta_{j1})\Big]\nonumber \\
&=\frac{|\nabla{f}|}{n-2}(\delta_{jk}\delta_{1i}-\delta_{ik}\delta_{j1})(\lambda_{1}+\lambda_{k}-\frac{S}{n-1}).
\end{align}

Using the assumption $\delta\WW(.,.,\nabla{f})=0$, we obtain that $$(P+\frac{1}{2(n-1)}Q)(\nabla{f}, .,.)=0.$$

Combining with (\ref{ran1}) yields,
\[P(1,k,k)=-\frac{1}{2(n-1)}Q(1,k,k)=\frac{\lambda_{1}|\nabla{f}|}{(n-1)}=\frac{|\nabla{f}|}{n-2}(\lambda_{1}+\lambda_{k}-\frac{S}{n-1}).\]

Thus $\lambda_{2}=\lambda_{3}=\lambda_{4}=\frac{S-\lambda_{1}}{n-1}$. Proposition \ref{Dequal0} then concludes the argument.\\


\end{proof}

The proof of Theorem \ref{deltaRc} follows from a similar argument. 

\begin{proof}{\bf (Theorem \ref{deltaRc})}

By equation (\ref{equaHdiv}),  $\delta (\Rc\circ g)=0$ implies $P-\frac{Q}{2}=0$. Thus, by Lemma \ref{product},
\[2|P|^2=2\left\langle{P, \frac{Q}{2}}\right\rangle=-\frac{|\nabla\SS|^2}{2}.\]

Hence $P=0=\nabla\SS$. It then follows from Corollary \ref{QD0} that $\delta\WW=\delta\SS=0$. The converse is obvious. 
\end{proof}

\begin{proof}{(\bf Theorem \ref{improved4})} 

Using a normal local frame, we can rewrite the assumption as,
\[ \sum_{i} f_{i}\WW^{+}_{ijkl}=0.\]
We pick an arbitrary index a and multiply both sides with $\WW_{ajkl}$ to arrive at, 
\[\sum_{i}f_{i}\WW_{ijkl}^{+}\WW^{+}_{ajkl}=0.\]
Applying identity (\ref{sumW}) yields,
\begin{align*}
0& =\sum_{jkl}\sum_{i}f_{i}\WW_{ijkl}^{+}\WW^{+}_{ajkl}\\
&=\sum_{i}f_{i}\sum_{jkl}\WW_{ijkl}^{+}\WW^{+}_{ajkl}\\
&=\sum_{i}f_{i}|\WW^{+}|^2 g_{ia}=f_{a}|\WW^{+}|^2.
 \end{align*} 
Since index a is arbitrary, we have $\nabla{f}=0$ or $|\WW^{+}|=0$.
\end{proof}

\begin{proof}({\bf Corollary \ref{cor5dim}}) 

By Theorem \ref{implyD} and Theorem \ref{deltaRc}, each condition implies $D\equiv 0$. Then, \cite[Lemma 4.2]{caochen11} further implies that $\delta \WW=0$. It follows, from classification results for harmonic Weyl tensor as discussed in the Introduction, that the manifold must be rigid. We now look at each case closely and observe that not all ranks can arise. 
 
{\bf i.} In this case, Lemma \ref{genvan} reveals that $\lambda_{0}-\lambda_{i}=0$ with $\Rc(\nabla{f})=\lambda_{0}\nabla_{f}$, and $\lambda_{i}$ is any other eigenvalue of $\Rc$. Therefore, the manifold structure must be Einstein.   

{\bf ii.} In this case, since $D\equiv 0$ implies $\Rc$ has at most two eigenvalues with one of multiplicity 1 and another of $n-1$. So k can only be $0, 1, n$.

 {\bf iii.} In this case, there is no obvious obstruction, so all rank can arise. 
   
\end{proof}

\begin{proof}({\bf Corollary \ref{cor4dim}})

 
The statement follows immediately from Theorem \ref{improved4}, \cite[Theorems 1.1, 1.2] {cw11}, and the analyticity of a GRS with bounded curvature \cite{bando87}.
\end{proof}

\section{\textbf{Appendix}}

In this appendix, we collect a few formulas that are related to this paper, they follow from direct computation.

\subsection{Conformal Change Calculation}In this subsection, we state the change of covariant derivative of the Weyl tensor and Bochner-Weitzenb\"ock type formula, with respect to the  conformal transformation of a metric. 
 
We first fix our notation. Let $(M^n,g)$ be a smooth Riemannian manifold and $u=e^f$ be a smooth positive function on M. A conformal change is defined by:
\begin{equation}
\label{conf}
\tilde{g}=e^{2f}=u^2 g.
\end{equation}
Then, for any tensor $\mathfrak{D}$ with respect to g, the corresponding for $\tilde{g}$ is denoted by $\widetilde{\mathfrak{D}}$. \\

We can calculate the transformation of the covariant derivative. For fixed $X,Y,Z$, 
\begin{align*} 
2e^{2f}(\widetilde{\nabla}_{X}Y,Z)_{g} =&2(\widetilde{\nabla}_{X}Y,Z)_{\tilde{g}}= X(Y,Z)_{\tilde{g}}+Y(Z,X)_{\tilde{g}}- Z(X,Y)_{\tilde{g}} \nonumber \\
&-(Y,[X,Z])_{\tilde{g}}-(Z[Y,X])_{\tilde{g}} +(X[Z,Y])_{\tilde{g}}
\\
=&2X(f)e^{2f}(Y,Z)_{g}+2Y(f)e^{2f}(Z,X)_g\\
&-2Z(f)e^{2f}(X,Y)_g+2e^{2f}(\nabla_{X}Y,Z)_{g}. \nonumber
\end{align*}
Thus, 
\begin{equation}
\label{confcov}
\widetilde{\nabla}_{X}Y=\nabla_{X}Y+X(f)Y+Y(f)X-(X,Y)_g\nabla f .
\end{equation}
Consequently, with the convention of $a \doteqdot \nabla^2 f-df\otimes df+\frac{1}{2}|\nabla f|^{2}g$, we have
\begin{align*}
\widetilde{\RR}=& e^{2f}\RR-e^{2f}a\circ g,
\\
\widetilde{\RR}^{l}_{ijk}&=\RR^{l}_{ijk}-a^{l}_{i}g_{jk}-a_{jk}\delta^{l}_{i}+a_{ik}\delta^{l}_{j}+a^{l}_{j}g_{ik},\nonumber\\
d\widetilde{\mu}=&e^{nf}d\mu,
\\
\widetilde{\triangle}h=&e^{-2f}\Big(\triangle h+(n-2)\nabla^{k}f\nabla_{k}h\Big),
\\
\widetilde{W}=& e^{2f}W,
\\
\widetilde{Rc}=& Rc-(n-2)a-\Big(\triangle f+\frac{n-2}{2}|\nabla{f}|^2\Big)g,
\\
\widetilde{\SS}=& e^{-2f}\Big(\SS-2(n-1)\triangle f-(n-2)(n-1)|\nabla f|^{2}\Big)
\\
=& e^{-2f}\Big(\SS-\frac{4(n-1)}{n-2}e^{-\frac{n-2}{2}f}\triangle (e^{\frac{n-2}{2}f})\Big) \text{ when $n>2$}.
\end{align*}

Now restricting our attention to dimension four, we arrive at
\begin{align*}
\widetilde{\SS}=& u^3(-6\Delta_{g}+\SS)u,\\
\widetilde{\WW}_{\tilde{a}\tilde{b}\tilde{c}\tilde{d}}&=u^{-4}\widetilde{\WW}_{abcd}=u^{-2}\WW_{abcd},\\
\widetilde{\Delta}&=u^{-2}(\Delta-2\frac{\nabla{u}}{u}\nabla),\\
\text{det}\widetilde{\WW_{+}}&=u^{-6}\text{det}\WW_{+}.
\end{align*}
\begin{lemma}
\label{divWconf} The divergence of the Weyl tensor under the above conformal change is given by,
\[\widetilde{\delta} \widetilde{\WW}(X,Y,Z)=\delta\WW(X,Y,Z)+(n-3)\WW(\frac{\nabla u}{u},X,Y,Z).\]
\end{lemma}

Next, we calculate the conformal change of the norm of the covariant derivative of the Weyl tensor.

\begin{lemma} Let $(M,g)$ be a four-dimensional Riemmanian manifold, and $\tilde{g}=u^2 g$, for some positive smooth function $u$. Then we have,
\label{confcov}
  \begin{equation}
 |\widetilde{\nabla} \widetilde{\WW}|^2=u^{-6}|\nabla \WW|^2+18u^{-8}|\nabla u|^2|\WW|^2-10u^{-7}\nabla{u}\nabla|\WW|^2+16\left\langle{\delta \WW, i_{\nabla u}\WW}\right\rangle .
 \end{equation}
\end{lemma}
\begin{proof}
We observe that, $$|\widetilde{\nabla}\widetilde{\WW}|^2=u^{-10}\Big((\widetilde{\nabla}_{{e}_{i}}\widetilde{\WW})_{{a}{b}{c}{d}}\Big)^2.$$ 

Then, 
\begin{align*}
(\widetilde{\nabla}_{{e}_{i}}\widetilde{\WW})_{{a}{b}{c}{d}}=&\nabla_{i}(u^{2}\WW_{abcd})-u^{2}\Big[\WW(\widetilde{\nabla}_{{e}_{i}}{a}, b, c, d)+\WW(a, \widetilde{\nabla}_{{e}_{i}}{b}, c, d)\\
&+\WW(a, b,\widetilde{\nabla}_{{e}_{i}} {c}, d)+\WW(a, b, c, \widetilde{\nabla}_{{e}_{i}}{d})\Big]\\
=&u^{2}\nabla_{i}\WW_{abcd}-2u u_{i}\WW_{abcd}+u\delta_{ia}\WW_{\nabla{u} bcd}-u\WW_{ibcd}u_{a}+u\delta_{ib}\WW_{a \nabla{u} cd}\\
&-u\WW_{aicd}u_{b}+u\delta_{ic}\WW_{ab \nabla{u}d}-u\WW_{abid}u_{c}+u\delta_{id}\WW_{abc \nabla{u}}-u\WW_{abci}u_{d}.
\end{align*}
Now we sum over all the index,  using Lemma \ref{sumW1}, we have
\begin{align*}
&(\nabla_{i}\WW_{abcd})^2 =|\nabla \WW|^2, &(u_{i}\WW_{abcd})^2= |\nabla u|^2 |\WW|^2,&\\
&(\delta_{ia}\WW_{\nabla{u} bcd})^2= 4(\WW_{\nabla{u} bcd})^2=4|\nabla u|^2 |\WW|^2,
&(\WW_{ibcd}u_{a})^2= |\nabla u|^2 |\WW|^2,&\\
&2\nabla_{i}\WW_{abcd}u_{i}\WW_{abcd}=\left\langle{\nabla |\WW|^2, \nabla u}\right\rangle,&
\nabla_{i}\WW_{abcd} \delta_{ia}\WW_{\nabla{u} bcd}=\left\langle{\delta \WW, i_{\nabla u}\WW}\right\rangle,&\\
&\nabla_{i}\WW_{abcd}\WW_{ibcd}u_{a}=\left\langle{\nabla |\WW|^2, \nabla u}\right\rangle-\left\langle{\delta \WW, i_{\nabla u}\WW}\right\rangle,&
u_{i}\WW_{abcd}\delta_{ia}\WW_{\nabla{u} bcd}=|\nabla u|^2 |\WW|^2,&\\
&u_{i}\WW_{abcd}\WW_{ibcd}u_{a}=|\nabla u|^2 |\WW|^2,&
\delta_{ia}\WW_{\nabla{u} bcd}\WW_{ibcd}u_{a}=|\nabla u|^2 |\WW|^2,&\\
&\delta_{ia}\WW_{\nabla{u} bcd}\delta_{ib}\WW_{a \nabla{u} cd}=-|\nabla u|^2 |\WW|^2,
&\delta_{ia}\WW_{\nabla{u} bcd}\WW_{aicd}u_{b}=0, &\\
&\WW_{\nabla{u} bcd}\delta_{ac}\WW_{ab \nabla{u}d}=\WW_{\nabla{u} bid}\WW_{bi d\nabla{u}}=\frac{1}{2}|\nabla u|^2 |\WW|^2,
&\WW_{ibcd}u_{a}\WW_{aicd}u_{b}=-|\nabla u|^2 |\WW|^2,&\\
&\WW_{ibcd}u_{a}\WW_{abid}u_{c}=\WW_{ib\nabla{u}d}\WW_{\nabla{u}bid}=\frac{1}{2}|\nabla u|^2 |\WW|^2.
\end{align*}
The result then follows immediately.

\end{proof}
We now can calculate the conformal change of Bochner-Weitzenb\"ock's formula.
\begin{corollary}
\label{confWeit1}
Let $(M,g)$ be a four-dimensional Riemmanian manifold, and $\tilde{g}=u^2 g$, for some positive smooth function $u$. If, \begin{equation*}
\label{confWeit2}
h=\Delta |\WW^{+}|^2-2|\nabla \WW^{+}|^2-\SS |\WW^{+}|^2+36\text{det}\WW^{+},
\end{equation*}
then 
\begin{equation*}
\label{confWeit}
u^6 \tilde{h}=h-20u^{-2}|\nabla{u}|^2|\WW^{+}|^2+2u^{-1}|\WW^{+}|^2 \Delta{u}+10u^{-1}\nabla{u}\nabla |\WW^{+}|^2-32u^{-1}\left\langle{\delta \WW^{+}, i_{\nabla u}\WW^{+}}\right\rangle.
\end{equation*}
\end{corollary}

\begin{proof}
With no confusion we denote $\WW\doteqdot \WW^{+}$ for simplicity and calculate that,
\begin{align*}
\widetilde{\Delta}|\widetilde{\WW}|^2=& \widetilde{\Delta}(u^{-4}|\WW|^2)=u^{-2}(\Delta (u^{-4}|\WW|^2)-2\frac{\nabla{u}}{u}\nabla(u^{-4}|\WW|^2) \\
=& u^{-2}\Big(u^{-4}\Delta|\WW|^2+|\WW|^2 \Delta u^{-4}+2\nabla u^{-4}\nabla |\WW|^2 \\
  & 2|\WW|^2 \frac{\nabla{u}}{u}\nabla u^{-4}-2u^{-4}\frac{\nabla u}{u}\nabla|\WW|^2 \Big)\\
=&u^{-6}\Delta|\WW|^2+20u^{-8}|\WW|^2 |\nabla u|^2-4u^{-7}|\WW|^2\Delta{u} \\
&-10u^{-7}\nabla{u}\nabla|\WW|^2+8u^{-8}|\nabla u|^2|\WW|^2\\
=&u^{-6}\Delta|\WW|^2+28u^{-8}|\WW|^2 |\nabla u|^2-4u^{-7}|\WW|^2\Delta{u}-10u^{-7}\nabla{u}\nabla|\WW|^2.\\
\widetilde{\SS}|\widetilde{\WW}|^2=&u^{-6}\SS|\WW|^2-6u^{-7}|\WW|^2 \Delta{u}.
\end{align*}
The result then follows by combining above equations with Lemma \ref{confcov}.
\end{proof}

\subsection{Along the Ricci Flow}
Inspired by the simplification in Bochner-Weitzenb\"ock formula in Theorem \ref{WSoliton}, we carry out a similar calculation on a Ricci solution. As a consequence, we obtain several interesting evolution equations involving the self-dual part and other components of the curvature operator. First, we state some useful lemmas.

\begin{lemma} 
\label{evolWplus}
Let $(M^4,g(t))$, $0\leq t<T\leq \infty$, be a solution to Ricci flow (\ref{RF}) and the curvature operator is decomposed as in (\ref{curdec}). Then,
\begin{equation}
\ppt \WW^{+}=\Delta\WW^{+}+ 2(\WW^{+})^2+4 (\WW^{+})^{\sharp}+2(CC^{T}-\frac{1}{3}|C|^2I^+).
\end{equation}
\end{lemma}

\begin{remark} Our convention agrees with \cite{mw93} but differs from \cite{HPCO}. 
\end{remark}

\begin{lemma}
For a four-dimensional Riemmanian manifold $(M, g)$, if the curvature is represented as in (\ref{curdec}), then, 
\begin{equation}
\label{WCC}
\left\langle{\WW^{+}, CC^T}\right\rangle=\frac{1}{4}\left\langle{\WW^+, \Rc\circ \Rc}\right\rangle.
\end{equation}
\end{lemma}

Using results above, we arrive at the following statement.

\begin{theorem}
\label{evolW2}
Let $(M^4,g(t)$, $0\leq t<T\leq \infty$, be a closed  solution to the Ricci flow (\ref{RF}), then we have following evolution equation, 
\begin{equation}
(\ppt -\Delta)|\WW^{+}|^2=-2|\nabla \WW^{+}|^2+36\text{det}_{\Lambda^2_{+}}\WW^{+} +\left\langle{\Rc\circ \Rc,\WW^{+}}\right\rangle.
\end{equation}
\end{theorem}

\begin{remark} The Weyl tensor is considered as the traceless part of the curvature operator (module out the Ricci and scalar components). Thus, it is interesting to compare the above calculation with the evolution equation for the traceless part of the Ricci curvature $h=|\EE|^2$ (see \cite{cao11}),
\begin{align*}
(\ppt -\Delta)h^2=&-2|\nabla\Rc|^2+\frac{|\nabla\SS|^2}{2}+\frac{2}{3}\SS h-4\EE^3+4\WW(\EE,\EE)\\
=&-2\nabla h\nabla (\ln\SS)-\frac{2}{\SS^2}|\SS \nabla \Rc-\Rc\nabla\SS|^2\\
&+2h^2(2|\nabla(\ln\SS)|^2+\frac{\SS}{3})-4\EE^3+4\WW(\EE,\EE).
\end{align*}

\end{remark}

A consequence of Theorem \ref{evolW2} is the following statement. 
\begin{corollary}
Let $(M,g(t))$, $0\leq t<T\leq \infty$, be a closed solution to the Ricci flow (\ref{RF}), then
\begin{align}
(\ppt -\Delta)(\frac{|\WW^{+}|^2}{\SS^2})=&-\frac{2}{\SS^4}|\SS\nabla \WW^{+}-\WW^{+}\nabla \SS|^2+\left\langle{\nabla{(\frac{|\WW^{+}|^2}{\SS^2})},\nabla \ln \SS^2}\right\rangle \nonumber\\
&+ 36\frac{\text{det}_{\Lambda^2_{+}}\WW^{+}}{\SS^2}+\frac{\left\langle{\Rc\circ \Rc,\WW^{+}}\right\rangle}{\SS^2}-4\frac{|\WW^{+}|^2|\Rc|^2}{\SS^3}.
\end{align}

\end{corollary}

\begin{remark}
On a GRS, the equation becomes
\begin{align}
-\Delta_{f}(\frac{|\WW^{+}|^2}{\SS^2})=&-\frac{2}{\SS^4}|\SS\nabla \WW^{+}-\WW^{+}\nabla \SS|^2+\left\langle{\nabla{(\frac{|\WW^{+}|^2}{\SS^2})},\nabla \ln \SS^2}\right\rangle \nonumber\\
&+ 36\frac{\text{det}_{\Lambda^2_{+}}\WW^{+}}{\SS^2}+\frac{\left\langle{\Rc\circ \Rc,\WW^{+}}\right\rangle}{\SS^2}-4\frac{|\WW^{+}|^2|\Rc|^2}{\SS^3}.
\end{align}
\end{remark}

An immediate application of the computation above and the maximum principle is the result below. 
\begin{proposition} Let $(M,g(t))$, $0\leq t<T\leq \infty$, be a closed solution to the Ricci flow (\ref{RF}). If $\text{det}_{\Lambda^2_{+}}\WW^{+}$ is nonpositive along the Ricci flow then there exists a constant $C=C(g(0))$, such that $\frac{|\WW^{+}|}{\SS}<C$ is preserved along the flow.
\end{proposition}


\def\cprime{$'$}
\bibliographystyle{plain}
\bibliography{bio}

\def\cprime{$'$}
\begin{thebibliography}{10}

\bibitem{ahs78}
M.~F. Atiyah, N.~J. Hitchin, and I.~M. Singer.
\newblock Self-duality in four-dimensional {R}iemannian geometry.
\newblock {\em Proc. Roy. Soc. London Ser. A}, 362(1711):425--461, 1978.

\bibitem{aubin76}
Thierry Aubin.
\newblock Équations différentielles non linéaires et problème de {Y}amabe
  concernant la courbure scalaire.
\newblock {\em J. Math. Pures Appl.}, 55:269--296, 1976.

\bibitem{bando87}
Shigetoshi Bando.
\newblock Real analyticity of solutions of {H}amilton's equation.
\newblock {\em Math. Z.}, 195(1):93--97, 1987.

\bibitem{berger61}
Marcel Berger.
\newblock Sur quelques vari\'et\'es d'{E}instein compactes.
\newblock {\em Ann. Mat. Pura Appl. (4)}, 53:89--95, 1961.

\bibitem{besse}
Arthur~L. Besse.
\newblock {\em Einstein manifolds}, volume~10 of {\em Ergebnisse der Mathematik
  und ihrer Grenzgebiete (3) [Results in Mathematics and Related Areas (3)]}.
\newblock Springer-Verlag, Berlin, 1987.

\bibitem{bohmwilking}
Christoph B{\"o}hm and Burkhard Wilking.
\newblock Manifolds with positive curvature operators are space forms.
\newblock {\em Ann. of Math. (2)}, 167(3):1079--1097, 2008.

\bibitem{b12rot}
Simon Brendle.
\newblock Rotational symmetry of self-similar solutions to the {R}icci flow.
\newblock {\em Invent. Math.}, 2013, doi: 10.1007/s00222-013-0457-0.

\bibitem{bs091}
Simon Brendle and Richard Schoen.
\newblock Manifolds with {$1/4$}-pinched curvature are space forms.
\newblock {\em J. Amer. Math. Soc.}, 22(1):287--307, 2009.

\bibitem{bs11}
Simon Brendle and Richard Schoen.
\newblock Curvature, sphere theorems, and the {R}icci flow.
\newblock {\em Bull. Amer. Math. Soc. (N.S.)}, 48(1):1--32, 2011.

\bibitem{bs072}
Simon Brendle and Richard~M. Schoen.
\newblock Classification of manifolds with weakly {$1/4$}-pinched curvatures.
\newblock {\em Acta Math.}, 200(1):1--13, 2008.

\bibitem{calabi58}
E.~Calabi.
\newblock An extension of {E}. {H}opf's maximum principle with an application
  to {R}iemannian geometry.
\newblock {\em Duke Math. J.}, 25:45--56, 1958.

\bibitem{caohd96}
Huai-Dong Cao.
\newblock Existence of gradient {K}\"ahler-{R}icci solitons.
\newblock In {\em Elliptic and parabolic methods in geometry (Minneapolis, MN,
  1994)}, pages 1--16. A K Peters, Wellesley, MA, 1996.

\bibitem{caohd09}
Huai-Dong Cao.
\newblock Recent progress on {R}icci solitons.
\newblock {\em Adv. Lect. Math.}, 11:1--38, 2009.

\bibitem{caochen12}
Huai-Dong Cao and Qiang Chen.
\newblock On locally conformally flat gradient steady {R}icci solitons.
\newblock {\em Trans. Amer. Math. Soc.}, 364(5):2377--2391, 2012.

\bibitem{caochen11}
Huai-Dong Cao and Qiang Chen.
\newblock On {B}ach-flat gradient shrinking {R}icci solitons.
\newblock {\em Duke Math. J.}, 162(6):1149--1169, 2013.

\bibitem{ein}
Xiaodong Cao.
\newblock Compact gradient shrinking {R}icci solitons with positive curvature
  operator.
\newblock {\em J. Geom. Anal.}, 17(3):425--433, 2007.

\bibitem{cao11}
Xiaodong Cao.
\newblock Curvature pinching estimate and singularities of the {R}icci flow.
\newblock {\em Comm. Anal. Geom.}, 19(5):975--990, 2011.

\bibitem{cw07}
Xiaodong Cao, Biao Wang, and Zhou Zhang.
\newblock On locally conformally flat gradient shrinking {R}icci solitons.
\newblock {\em Commun. Contemp. Math.}, 13(2):269--282, 2011.

\bibitem{ca12}
Giovanni Catino.
\newblock Generalized quasi-{E}instein manifolds with harmonic {W}eyl tensor.
\newblock {\em Math. Z.}, 271(3-4):751--756, 2012.

\bibitem{cm11}
Giovanni Catino and Carlo Mantegazza.
\newblock The evolution of the {W}eyl tensor under the {R}icci flow.
\newblock {\em Ann. Inst. Fourier (Grenoble)}, 61(4):1407--1435 (2012), 2011.

\bibitem{cw11}
Xiuxiong Chen and Yuangi Wang.
\newblock On four-dimensional anti-self-dual gradient {R}icci solitons.
\newblock {\em preprint}, 2011.

\bibitem{chowluni}
Bennett Chow, Peng Lu, and Lei Ni.
\newblock {\em Hamilton's {R}icci flow}, volume~77 of {\em Graduate Studies in
  Mathematics}.
\newblock American Mathematical Society, Providence, RI, 2006.

\bibitem{Derd83}
Andrzej Derdzi{\'n}ski.
\newblock Self-dual {K}\"ahler manifolds and {E}instein manifolds of dimension
  four.
\newblock {\em Compositio Math.}, 49(3):405--433, 1983.

\bibitem{Derd06}
Andrzej Derdzi{\'n}ski.
\newblock A {M}yers-type theorem and compact {R}icci soliton.
\newblock {\em Proc. Amer. Math. Soc.}, 134(12):3645--3648, 2006.

\bibitem{fik03}
Mikhail Feldman, Tom Ilmanen, and Dan Knopf.
\newblock Rotationally symmetric shrinking and expanding gradient
  {K}\"ahler-{R}icci solitons.
\newblock {\em J. Differential Geom.}, 65(2):169--209, 2003.

\bibitem{fega11}
Manuel Fern{\'a}ndez-L{\'o}pez and Eduardo Garc{\'{\i}}a-R{\'{\i}}o.
\newblock Rigidity of shrinking {R}icci solitons.
\newblock {\em Math. Z.}, 269(1-2):461--466, 2011.

\bibitem{g98}
Matthew~J. Gursky.
\newblock The {W}eyl functional, de {R}ham cohomology, and
  {K}\"ahler-{E}instein metrics.
\newblock {\em Ann. of Math. (2)}, 148(1):315--337, 1998.

\bibitem{g00}
Matthew~J. Gursky.
\newblock Four-manifolds with {$\delta W^+=0$} and {E}instein constants of the
  sphere.
\newblock {\em Math. Ann.}, 318(3):417--431, 2000.

\bibitem{gl99}
Matthew~J. Gursky and Claude Le{B}run.
\newblock On {E}instein manifolds of positive sectional curvature.
\newblock {\em Ann. Global Anal. Geom.}, 17(4):315--328, 1999.

\bibitem{H3}
Richard~S. Hamilton.
\newblock Three-manifolds with positive {R}icci curvature.
\newblock {\em J. Differential Geom.}, 17(2):255--306, 1982.

\bibitem{HPCO}
Richard~S. Hamilton.
\newblock Four-manifolds with positive curvature operator.
\newblock {\em J. Differential Geom.}, 24(2):153--179, 1986.

\bibitem{Hsurface}
Richard~S. Hamilton.
\newblock The {R}icci flow on surfaces.
\newblock In {\em Mathematics and general relativity ({S}anta {C}ruz, {CA},
  1986)}, volume~71 of {\em Contemp. Math.}, pages 237--262. Amer. Math. Soc.,
  Providence, RI, 1988.

\bibitem{kb08}
Brett Kotschwar.
\newblock On rotationally invariant shrinking {R}icci solitons.
\newblock {\em Pacific J. Math.}, 236(1):73--88, 2008.

\bibitem{lp87}
John~M. Lee and Thomas~H. Parker.
\newblock The {Y}amabe problem.
\newblock {\em Bull. Amer. Math. Soc. (N.S.)}, 17(1):37--91, 1987.

\bibitem{mw93}
Mario~J. Micallef and McKenzie~Y. Wang.
\newblock Metrics with nonnegative isotropic curvature.
\newblock {\em Duke Math. J.}, 72(3):649--672, 1993.

\bibitem{munse09}
Ovidiu Munteanu and Natasa Sesum.
\newblock On gradient {R}icci solitons.
\newblock {\em J. Geom. Anal.}, 23(2):539--561, 2013.

\bibitem{naber07}
Aaron Naber.
\newblock Noncompact shrinking four solitons with nonnegative curvature.
\newblock {\em J. Reine Angew. Math.}, 645:125--153, 2010.

\bibitem{niwa08}
Lei Ni and Nolan Wallach.
\newblock On a classification of gradient shrinking solitons.
\newblock {\em Math. Res. Lett.}, 15(5):941--955, 2008.

\bibitem{perelman1}
Grigori Perelman.
\newblock The entropy formula for the {R}icci flow and its geometric
  applications.
\newblock {\em preprint}, 2002.

\bibitem{perelman2}
Grigori Perelman.
\newblock Ricci flow with surgery on three-manifolds.
\newblock {\em preprint}, 2003.

\bibitem{pe06book}
Peter Petersen.
\newblock {\em Riemannian geometry}, volume 171 of {\em Graduate Texts in
  Mathematics}.
\newblock Springer, New York, second edition, 2006.

\bibitem{pewy09}
Peter Petersen and William Wylie.
\newblock Rigidity of gradient {R}icci solitons.
\newblock {\em Pacific journal of mathematics}, 241(2):329--345, 2009.

\bibitem{pewy10}
Peter Petersen and William Wylie.
\newblock On the classification of gradient {R}icci solitons.
\newblock {\em Geom. Topol.}, 14(4):2277--2300, 2010.

\bibitem{prs09}
Stefano Pigola, Michele Rimoldi, and Alberto~G. Setti.
\newblock Remarks on non-compact gradient {R}icci solitons.
\newblock {\em Math. Z.}, 268(3-4):777--790, 2011.

\bibitem{schoen84}
Richard Schoen.
\newblock Conformal deformation of a {R}iemannian metric to constant scalar
  curvature.
\newblock {\em J. Diff. Geom.}, 20:479--495, 1984.

\bibitem{st69}
I.~M. Singer and J.~A. Thorpe.
\newblock The curvature of {$4$}-dimensional {E}instein spaces.
\newblock In {\em Global {A}nalysis ({P}apers in {H}onor of {K}. {K}odaira)},
  pages 355--365. Univ. Tokyo Press, Tokyo, 1969.

\bibitem{wu88}
Hung~Hsi Wu.
\newblock The {B}ochner technique in differential geometry.
\newblock {\em Math. Rep.}, 3(2):i--xii and 289--538, 1988.

\bibitem{xu12}
Gouyi Xu.
\newblock Four dimensional shrinking gradient solitons with small curvature.
\newblock {\em preprint}, 2012.

\bibitem{yangdg00}
DeGang Yang.
\newblock Rigidity of {E}instein four-manifolds with positive curvature.
\newblock {\em Invent. Math}, 142:435--450, 2000.

\bibitem{zhangzh09}
Zhu-Hong Zhang.
\newblock Gradient shrinking solitons with vanishing {W}eyl tensor.
\newblock {\em Pacific J. Math.}, 242(1):189--200, 2009.

\end{thebibliography}

\end{document}